\documentclass[a4paper,twoside,10pt]{amsart}
\usepackage{amsmath,amsfonts,amssymb,amsthm,mathtools}
\allowdisplaybreaks
\usepackage{todonotes}
\newtheorem{theorem}{Theorem}[section]
\newtheorem{lemma}{Lemma}[section]

\newtheorem{corollary}{Corollary}[section]
\usepackage{color,graphicx, float}
\usepackage[colorlinks=true,pdfstartview=FitH,pdfview=FitH]{hyperref}
\numberwithin{equation}{section}

\newcommand\mytop[2]{\genfrac{}{}{0pt}{}{#1}{#2}}
\def\expe{{\rm e}}
\def\iunit{{\rm i}}

\def\qpr#1#2#3{\left(#1;#2\right)_{#3}}
\def\Eq#1{{\rm E}_q\left(#1\right)}
\def\bigO{{\mathcal O}}

\def\qphi#1#2#3#4#5#6{{{}_{#1}\phi_{#2}}\!\left(\genfrac{}{}{0pt}{}{#3}{#4};#5;#6\right)}

\def\qpsi#1#2#3#4#5#6{{{}_{#1}\psi_{#2}}\!\left(\genfrac{}{}{0pt}{}{#3}{#4};#5;#6\right)}
\def\qhat{{\hat{q}}}
\def\Thetaq#1{{\rm \theta}_q\left(#1\right)}
\def\dThetaq#1{{\rm \theta}'_q\left(#1\right)}
\def\Pq#1{\mathcal{P}_q\left(#1\right)}
\def\dPq#1{\mathcal{P}'_q\left(#1\right)}
\def\Cq{{C_q}}
\def\Uq#1#2#3{{\rm U}_q\left(#1,#2;#3\right)}
\def\UqK#1#2#3#4{{\rm U}_{q,#1}\left(#2,#3;#4\right)}
\def\intd{{\rm d}}
\def\LqK#1{\mathcal{L}_{q,#1}}
\def\Lq{{\mathcal{L}_q}}
\def\pK#1#2#3{p_{#1}\left(#2,#3\right)}
\def\K#1{\kappa\left(#1\right)}
\def\KK{{\kappa}}

\def\Psiq#1{\Psi_q\left(#1\right)}
\def\ee{e}

\title[The $q$-Laplace Transforms compared: ${}_2\phi_0$]{The $q$-Laplace Transforms compared: 
the basic confluent hypergeometric function ${}_2\phi_0$}
\author{Daniel Meikle}
\author{Adri Olde Daalhuis}
\address{School of Mathematics and Maxwell Institute for Mathematical Sciences, University of Edinburgh, Peter Guthrie Tait Road, Edinburgh, EH9 3FD, UK}
\email{D.Meikle@ed.ac.uk}
\email{A.OldeDaalhuis@ed.ac.uk}
\date{\today}

\begin{document}
\begin{abstract}
In solving $q$-difference equations, and in the definition of $q$-special functions, we encounter formal power series 
in which the $n$th coefficient is of size $q^{-\binom{n}{2}}$ with $q\in(0,1)$ fixed. To make sense of
these formal series, a $q$-Borel-Laplace resummation is required. There are three candidates for the
$q$-Laplace transform, resulting in three different resummations. Surprisingly, the differences between these
resummations have hardly been discussed in the literature. 

Our main result provides explicit formulas for these $q$-exponentially small differences. 
We also give simple Mellin--Barnes integral representations for all the basic hypergeometric
${}_r\phi_s$ functions and derive a third (discrete) orthogonality condition for the Stieltjes--Wigert polynomials.

As the main application, we introduce three resummations for the ${}_2\phi_0$ function which can be seen as $q$ versions
of the Kummer $U$ function. We derive many of their properties, including interesting integral and sum
representations, connection formulas, and error bounds.
\end{abstract}

\maketitle


\section{Introduction and Summary}
Solving differential or difference equations via series solutions often leads to Gevrey 1 divergent series of the 
form $\hat{w}(z)=\sum_n a_nz^n$. These series have coefficients that grow like 
$\frac{K\Gamma(n+\alpha)}{F^{n+\alpha}}$ as $n\to\infty$, resulting in a growth rate of “factorial over power”.

One way to make sense of these formal series is the introduction of the Borel transform:
$B(t)=\sum_n\frac{a_n}{\Gamma(n+\alpha)} t^n$. This infinite series has a finite radius of convergence, 
and the singularities in the finite Borel plane (the complex $t$-plane) cause the Stokes phenomenon. 
Each singularity corresponds to a specific exponentially-small scale.

To return to the $z$-plane, one takes the Laplace transform, defined as 
$w_B(z)=z^{-\alpha}\int_0^\infty \expe^{-t/z}B(t)t^{\alpha-1}\intd t$. Note that any function that will have
$w(z)\sim \sum_n a_nz^n$ as $z\to 0$ in a complex sector of opening more than $\pi$ will be uniquely
determined by this divergent series, that is, $w(z)=w_B(z)$.
For more details, see, for example, \cite{LR2016,Mitschi2016}.

In contrast to Gevrey 1 asymptotics, solving $q$-difference equations, particularly in the definition of 
basic hypergeometric functions, involves a different type of divergence. 
In these problems, there will be an additional parameter $q$ and we assume $|q|<1$. In Section \ref{S:SpecialFunct}, we introduce all the special notation. This section contains many identities for theta functions.

The typical growth of coefficients in $q$-difference equations is $a_n\sim Kq^{-\frac12 n(n-1)}/F^n$, as $n\to\infty$. 
A series with these coefficients is of $q$-Gevrey order 1, see \cite{Ramis:1992:GEF}.
This divergence is much stronger, necessitating the use of different Borel and Laplace transforms. 

The $q$-Borel transform is, obviously, $B_q(t)=\sum_n a_n q^{\frac12 n(n-1)} t^n$. 
In interesting cases, it has a finite radius of convergence. The singularities in the Borel plane still 
cause the Stokes phenomenon, but now all the finite singularities are linked to the same exponentially-small scale! 
For the corresponding $q$-Laplace transform, we need to introduce alternatives for $\expe^{-t}$. They are
\begin{equation}\label{Eq}
    \Eq{\tau}=\exp\left(\frac{\ln^2(\tau/\sqrt{q})}{2\ln q}\right)=q^{\frac{1}{8}}\tau^{-\frac{1}{2}}\exp\left(\frac{\ln^2\tau}{2\ln q}\right),
\end{equation}
and the reciprocal of the theta function 
\begin{equation}\label{Thetaq}
    \Thetaq{\tau}=\qpr{q,-\tau,-q/\tau}{q}{\infty}=\sum_{n=-\infty}^{\infty}q^{\binom{n}{2}}\tau^n.
\end{equation}
Note that $\Eq{q^x}=q^{\frac12\left(x-\frac12\right)^2}$. Terms of size $\Eq{\tau}$ are defined to be
`$q$-exponentially small'. The growth of solutions of $q$-difference equations is discussed in \cite{Ramis:1992:GEF}.

In the last several decades, many versions of the $q$-Laplace transform have been introduced in the literature, see for example \cite{Tahara2017}.
We require the property that $t^n\mapsto q^{-\binom{n}{2}}z^n$. The three main $q$-Laplace transforms
with this property are the two continuous $q$-Laplace transforms defined in \cite{Ramis:1993:SDA}, \cite{Zhang1999} and \cite{Zhang2000}:
\begin{equation}\label{qLaplace1}
    (\LqK{E} B)(z)=\Cq\int_0^\infty B(t)\Eq{\frac{t}{z}}\frac{\intd t}{t}, \qquad (\LqK{\theta} B)(z)=\frac{-1}{\ln q}\int_0^\infty \frac{B(t)}{\Thetaq{\frac{t}{z}}}\frac{\intd t}{t},
\end{equation}
(with $\Cq=1/\sqrt{-2\pi\ln q}$) and the discrete $q$-Laplace transform, defined in \cite{Zhang2002}:
\begin{equation}\label{qLaplace2}
    (\LqK{\lambda} B)(z)=\sum_{n=-\infty}^{\infty}\frac{B(q^n\lambda)}{\Thetaq{\frac{q^n\lambda}{z}}},
\end{equation}
where $z\in\mathbb{C}\backslash\{-q^n\lambda;n\in\mathbb{Z}\}$.
There is no similar discrete transform in terms of the $\Eq{\tau}$ function instead of $1/\Thetaq{\tau}$.
Note that the Borel transform $B(t)$ will have singularities in the complex $t$-plane, hence,
$q^n\lambda$ should be bounded away from these singularities. In many examples, $B(t)$ has poles at
$t=-q^{-m}$, $m=0,1,2,\ldots$, and we require $\lambda\in\mathbb{C}\backslash\{-q^n;n\in\mathbb{Z}\}$. 

\begin{table}[ht]
\begin{center}
\begin{tabular}{|c|c|c|}
\hline
&$f(t)$ & $F(z)=(\LqK{\ee} f)(z)$ \\
\hline

(1)&$t^nf(t)$ & $\displaystyle q^{-\binom{n}{2}} z^nF(zq^{-n})$\\

(2)&$f(qt)$ & $\displaystyle F(qz)$\\

(3)&$f(q/t)$ & $\displaystyle F(1/z)$\\

\hline

(4)&$t^n$ & $\displaystyle q^{-\binom{n}{2}} z^n$\\

(5)&$\qpr{-at}{q}{\infty}$ & $\displaystyle\frac1{\qpr{az}{q}{\infty}}$\\

(6)&$\displaystyle \qpr{-bq/t}{q}{\infty}$ & $\displaystyle \frac{1}{\qpr{b/z}{q}{\infty}}$\\

(7)&$\displaystyle \qpr{-at}{q}{\infty}\qpr{-bq/t}{q}{\infty}$ & 
$\displaystyle \frac{\qpr{ab}{q}{\infty}}{\qpr{az}{q}{\infty}\qpr{b/z}{q}{\infty}}$\\

(8)&$\displaystyle \frac{\qpr{-at}{q}{\infty}}{\qpr{-bt}{q}{\infty}}$ & 
$\displaystyle \UqK{\ee}{0}{\frac{a}{b}}{bz}$,~~$\ee=E,\theta$\\

(9)& $\displaystyle \qphi{1}{1}{q^{-n}}{0}{q}{-tq^n}$ & $\qpr{z}{q}{n}$ \\

(10)& $\displaystyle \qphi{r}{s+1}{a_1,\ldots,a_r}{b_1,\ldots,b_s,0}{q}{-t}$ & $\displaystyle \qphi{r}{s}{a_1,\ldots,a_r}{b_1,\ldots,b_s}{q}{z}$ \\

\hline
\end{tabular}
\caption{$q$-Laplace transforms. For (5), (6) and (7) we require $b<z<a^{-1}$. Entry (7) follows from
\cite[Theorem 6.3]{Cohl2025}, and (8) follows from \eqref{UqSymmetric}. For (9), note that the Stieltjes--Wigert
polynomials are $S_n(x;q)=\frac{1}{\qpr{q}{q}{n}}\qphi{1}{1}{q^{-n}}{0}{q}{-xq^{n+1}}$ with orthogonality relations
\eqref{SW1} and \eqref{SW3}.}
\label{tab:qLaplace}
\end{center}
\end{table}%

Having three $q$-Laplace transforms will also give us three different $q$-Borel-Laplace resummations of the
original divergent series. It is very surprising that these differences have not been discussed in the literature for a general series. In \cite{DiVizio2009}, the difference between $(\LqK{\theta} B)(z)$ and $(\LqK{\lambda} B)(z)$ is computed for the specific case of a divergent $q$-Euler series.
Is there a best $q$-Borel-Laplace resummation? Note that we also lose the uniqueness property: all three
resummations will have the original divergent series as its asymptotic series 
as $z\to 0$ in complex sectors of opening more than $\pi$.

It should be obvious from these observations that the difference of two of these $q$-Borel-Laplace resummations
is $q$-exponentially small, that is, it should be of size smaller than any term in the original divergent series.
To be more precise, the difference is of size
$\Eq{z}$. The growth of this elementary function
lies in between algebraic and exponential as $z\to 0$ and $z\to \infty$.
In Section \ref{S:Laplace}, we give exact formulas for the differences. They are in terms of the functions
that are being switched on via the Stokes phenomenon. Consideration of the case of a $q$-Euler series with our results gives agreement with those presented in \cite{DiVizio2009}.

The first $q$-Borel-Laplace resummation $(\LqK{E} B)(z)$ is in terms of the elementary function $\Eq{z}$. 
The resummation is a multivalued function, having the required asymptotic expansions 
$(\LqK{E} B)(z)\sim\sum_n{a_n}z^n$ for \emph{all} $\arg z$. 
The main reason for this is that, in contrast to $\exp(-1/z)$, the function $\Eq{z}$ is exponentially small
for all $\arg z$.

The second $q$-Borel-Laplace resummation $(\LqK{\theta} B)(z)$ is in terms of a reciprocal $\theta_q$ function. 
Again, the resummation is multivalued, but now $(\LqK{\theta} B)(z)\sim\sum_n{a_n}z^n$ 
holds for $|\arg z|<2\pi$, because the reciprocal $\theta_q$ function
is switched on via the Stokes phenomenon and it has poles along $\arg z=\pm2\pi$.

Finally, the third $q$-Borel-Laplace resummation $(\LqK{\lambda} B)(z)$ has a free parameter $\lambda$. It is in terms
of an infinite series of reciprocal $\theta_q$ functions. This resummation is one-valued, and we have
$(\LqK{\lambda} B)(z)\sim\sum_n{a_n}z^n$ for $|\arg (z/\lambda)|<\pi$, because of the poles of the reciprocal
$\theta_q$ function.
The $(\LqK{\theta} B)(z)$ resummation may in fact be obtained from the $(\LqK{\lambda} B)(z)$ resummation by integrating over the free parameter $\lambda$ via
\begin{equation}\label{Intlamda}
    (\LqK{\theta} B)(z)=\frac{-1}{\ln q}\int_q^1 (\LqK{\lambda} B)(z)\,\frac{\intd \lambda}{\lambda},
\end{equation}
\cite[Theorem 4.14]{DiVizio2009}, and as a consequence
\begin{equation}\label{IntlamdaE}
    (\LqK{E} B)(z)=\frac{-1}{\ln q}\int_q^1 (\LqK{\lambda} B)(z)\,
    \Pq{\frac{\lambda}{z}}\frac{\intd \lambda}{\lambda},
\end{equation}
with $\Pq{\tau}$ defined in \eqref{qperiodic}.

One could argue that $(\LqK{\lambda} B)(z)$ is the most natural resummation because it is one-valued, but it has
poles in an artificial direction. If somehow one could link these poles to the original problem ($q$-difference equation)
then this could be the most natural resummation. The Mellin--Barnes integral representation \eqref{rphisMB_E}
for the basic hypergeometric functions, however, suggest that $(\LqK{E} B)(z)$ is the most natural.
Moreover, in a sequel to this paper, we will discuss Mellin--Barnes integral representations for all ${}_r\phi_s$. 
In that paper, it will be demonstrated that in the case of $q$-Gevrey series of order $k$ with $k>1$
(needed when $r-s-1=k>1$), the obvious $(\mathcal{L}_{q^k,\lambda}B)(z)$ resummation does not always solve the
original $q$-difference equation.

These $q$-Laplace transforms appear in the orthogonality relations 
\cite[\href{https://dlmf.nist.gov/18.27.vi}{\S 18.27.vi}]{NIST:DLMF}
for the Stieltjes--Wigert polynomials,
\begin{equation}\label{SW1}
    \Cq\int_0^\infty S_n(x;q)S_m(x;q)\Eq{x}\intd x=
    \frac{-1}{\ln q}\int_0^\infty \frac{S_n(x;q)S_m(x;q)}{\Thetaq{x}}\intd x=\frac{\delta_{n,m}}{q^n\qpr{q}{q}{n}}.
\end{equation}
In the literature, it is observed that it is remarkable that the measure is not unique. From the viewpoint
of $q$-Laplace transforms, this is not remarkable, because the three $q$-Laplace transforms have
property (4) of Table \ref{tab:qLaplace}, that is, they have the same moments.
Hence, applying each of the $q$-Laplace transforms to a polynomial
should give the same result. It follows that there is a third (discrete) orthogonality relation
\begin{equation}\label{SW3}
    \sum_{k=-\infty}^\infty\frac{S_n(q^k\lambda;q)S_m(q^k\lambda,q)q^k\lambda}{\Thetaq{q^k\lambda}}
    =\frac{\lambda}{1-q}\int_0^\infty \frac{S_n(x\lambda;q)S_m(x\lambda,q)}{\Thetaq{x\lambda}}\intd_q x
    =\frac{\delta_{n,m}}{q^n\qpr{q}{q}{n}},
\end{equation}
in which $\lambda$ is a free parameter. The second representation in \eqref{SW3} is just the first, but using the
$q$-integral notation \cite[\href{https://dlmf.nist.gov/17.2.E48}{17.2.48}]{NIST:DLMF}.

An important class of special functions are the ones that can be expressed in terms of the Kummer $U$ function.
See \cite[\href{https://dlmf.nist.gov/13.6}{\S 13.6}]{NIST:DLMF}. It has four important integral representations:
\begin{equation}\begin{split}\label{KummerU0}
     U(a,b;z)&=z^{1-a}\int_{0}^{\infty}\expe^{-zt}
     {}_2F_1\left(\mytop{a,a-b+1}{1};-t\right)  \intd t\\
     &=\frac{1}{\Gamma(a)}\int_{0}^{\infty}\expe^{-zt}t^{a-1}(1+t)^{b-a-1}\intd t\\
     &=\frac{z^{1-a}}{\Gamma(a)\Gamma(1+a-b)}\int_{0}^{\infty}
     \frac{U(b-a,b;t) \expe^{-t}t^{a-1}}{t+z} \intd t\\
     &= \frac{z^{-a}}{2\pi\iunit}\int_{-\iunit\infty}^{\iunit\infty}
    \frac{\Gamma(a+t)\Gamma(1+a-b+t)\Gamma(-t)}{\Gamma(a)\Gamma(1+a-b)}z^{-t}\intd t.
\end{split}\end{equation}
The first integral is the Borel-Laplace resummation of the divergent generalised hypergeometric series ${}_2F_0$,
compare \cite[\href{https://dlmf.nist.gov/13.6.E21}{13.6.21}]{NIST:DLMF}. The second integral representation is
probably the best known and is in terms of elementary functions. The third one is the Cauchy-Heine integral 
representation, and is very useful for error bounds and exponentially improved asymptotics, 
see \cite{OldeDaalhuis:1995:HSSa}. The fourth one is one of the Mellin--Barnes integral representations
and it contains the connection between the local behaviours near $z=\infty$ and $z=0$.

Hence, the $q$-Borel-Laplace resummation of 
the basic hypergeometric series $\qphi{2}{0}{a,b}{-}{q}{z}$ should also be an important class of functions.
Note that this formal series satisfies the $q$-difference equation
\begin{equation} \label{2phi0DiffEq}
    zy(zq^{-2})+(q-(a+b)z)y(zq^{-1})-(q-abz)y(z)=0.
\end{equation}
In the second half of the paper, we study the three $q$-Borel-Laplace transforms of $\qphi{2}{0}{a,b}{-}{q}{z}$.
In Section \ref{S:Application}, we introduce the three functions. 
They are denoted by $\UqK{\ee}{a}{b}{z}$, in which $\ee$ is either $E$, $\theta$ or $\lambda$.
In Section \ref{S:integrals}, we will show that each of the representations in \eqref{KummerU0} has a 
corresponding integral representation for $\UqK{\ee}{a}{b}{z}$ with $\ee=E,\theta$, and that, in addition, 
$\UqK{\lambda}{a}{b}{z}$ has similar sum representations. The Mellin--Barnes
integral representations particularly show that the first $q$-Borel-Laplace resummation $\UqK{E}{a}{b}{z}$ is the most natural.
The solutions of \eqref{2phi0DiffEq} near $z=\infty$ are in terms of convergent ${}_2\phi_1$ functions.
In Section \ref{S:connections}, we derive the connection relations between $\UqK{\ee}{a}{b}{z}$
and these ${}_2\phi_1$ functions. Relatively sharp error bounds are given in Section \ref{S:errorbounds},
and other properties (recurrence relations w.r.t. the parameters $a,b$) are given in Section \ref{S:other}.
In that final section, we do also include an attempt to obtain a continued fraction representation.
This continued fraction representation does not depend on the choice of $\ee$, and if it were convergent, 
it might have resolved the question of what is the most natural resummation. Unfortunately, our
continued fraction representation is not convergent.

The proofs of many of the theorems are relatively long and technical, and are moved to the Appendix.


\section{$q$-Special Functions}\label{S:SpecialFunct}
We take $0<q<1$, $\qhat=\expe^{\frac{2\pi^2}{\ln q}}$ and define the $q$-Pochhammer symbol as $\qpr{a}{q}{\infty}=\Pi_{n=0}^\infty (1-aq^n)$ with $\qpr{a}{q}{\nu}=\qpr{a}{q}{\infty}/\qpr{aq^\nu}{q}{\infty}$, $a,\nu\in\mathbb{C}$. For a positive integer $n\in\mathbb{N}$, we therefore have $\qpr{a}{q}{n}=(1-a)(1-aq)\ldots(1-aq^{n-1})$ and $\qpr{a}{q}{-n}=1/\qpr{aq^{-n}}{q}{n}$. We also use the shorthand $\qpr{a_1,a_2,\ldots,a_k}{q}{n}=\qpr{a_1}{q}{n}\qpr{a_2}{q}{n}\ldots\qpr{a_k}{q}{n}$.

We introduce the basic hypergeometric series,
\begin{equation}\label{qphi}
    \qphi{r}{s}{a_1,a_2,\ldots,a_r}{b_1,b_2,\ldots,b_s}{q}{z}=\sum_{n=0}^\infty \frac{\qpr{a_1,a_2,\ldots,a_r}{q}{n}}{\qpr{b_1,b_2,\ldots,b_s,q}{q}{n}}\left((-1)^nq^{\binom{n}{2}}\right)^{1+s-r}z^n,
\end{equation}
\begin{equation}\label{qpsi}
    \qpsi{r}{s}{a_1,a_2,\ldots,a_r}{b_1,b_2,\ldots,b_s}{q}{z}=\sum_{n=-\infty}^\infty \frac{\qpr{a_1,a_2,\ldots,a_r}{q}{n}}{\qpr{b_1,b_2,\ldots,b_s}{q}{n}}\left((-1)^nq^{\binom{n}{2}}\right)^{s-r}z^n.
\end{equation}
The radius of convergence for the basic hypergeometric series $_r\phi_s$ are $\infty$, $1$ and $0$ for $r<s+1$, $r=s+1$ and $r>s+1$ respectively.
Regarding the $r>s+1$ case: In \cite{Adachi2019} it is demonstrated that these formal series are definitely 
$q$-Borel summable. It is a very good source for more details for the general case.
For the bilateral sum $_r\psi_s$, the series converges for $|(b_1\ldots b_s)/(a_1\ldots a_rz)|<1$ when $r\leq s$ and also for $|z|<1$ when $r=s$. We also require that the $a_j$ do not take any of the values $q^{n+1}$ and the $b_j$ do not take any of the values $q^{-n}$.

Theorem \ref{thmMB} gives us Mellin--Barnes integral representations for the three resummations of ${}_2\phi_0$.
In a sequel to this paper, we will discuss Mellin--Barnes integral representations for all ${}_r\phi_s$.
The $E$-version is given below. Comparing it with \eqref{qphi} demonstrates that the $E$-resummation is the most natural resummation.
\begin{equation}
\begin{split}\label{rphisMB_E}
    \qphi{r}{s}{a_1,\ldots,a_r}{b_1,\ldots,b_s}{q}{z}=
    \frac{-1}{2\pi\iunit}\int_{\mathcal{C}}\frac{\qpr{a_1,\ldots,a_r}{q}{\tau}}{\qpr{b_1,\ldots,b_s,q}{q}{\tau}}\frac{\pi(-z)^\tau}{\sin\pi\tau}\left((-1)^\tau q^{\binom{\tau}{2}}\right)^{1+s-r}\intd\tau.
\end{split}
\end{equation}
The contour $\mathcal{C}$ separates the poles at $\tau=0,1,2,\ldots$, from the poles at
$\tau+\frac{\ln a_j}{\ln q}=-m-n\frac{2\pi\iunit}{\ln q}$, $m\in\mathbb{N}_0$, and $n\in\mathbb{Z}$.
In the case $s\leq r-1$, the contour $\mathcal{C}$ is taken to run from $-\iunit\infty$ to $\iunit\infty$,
and in the case $s>r-1$, it starts at $-\iunit d+\infty$, encircles the origin once in the negative sense,
and returns to $\iunit d+\infty$ with $d>0$. These contours may be seen pictorially in cases (i) and (ii) in 
\cite[\href{https://dlmf.nist.gov/16.17.F1}{Fig.~16.17.1}]{NIST:DLMF}. 
This contour integral representation is given in \cite[(4.5.1)]{Gasper2004} for the case of $s=r-1$.
By pushing the contour over the sequence of the poles along the positive real axis, we see that we obtain 
the series representation \eqref{qphi}. For $s\leq r-1$, we may also push the integration contour over 
the poles in the negative real half-plane to obtain a connection formula with the solutions at infinity.

The $q$-Laplace transforms were introduced in \eqref{qLaplace1} and \eqref{qLaplace2}.
Both $F(\tau)=\Eq{\tau}$ and $F(\tau)=1/\Thetaq{\tau}$ have the properties
\begin{equation} \label{Fprop}
    F(\tau^{-1})=\tau F(\tau), \qquad F(q^n\tau)=\tau^n q^{\binom{n}{2}}F(\tau).
\end{equation}
The second property above holds for $n\in\mathbb{C}$ for $\Eq{\tau}$ and $n\in\mathbb{Z}$ for $1/\Thetaq{\tau}$. With $\Cq=1/\sqrt{-2\pi\ln q}$, we have
\begin{equation}\label{BLpower}
    \Cq \int_0^\infty t^{n-1}\Eq{\frac{t}{z}}\intd t=q^{-\binom{n}{2}}z^n,\qquad \frac{-1}{\ln q}\int_0^\infty \frac{t^{n-1}}{\Thetaq{\frac{t}{z}}} \intd t=q^{-\binom{n}{2}}z^n.
\end{equation}
From the observation above, we see that the first equation in \eqref{BLpower} holds for $n\in\mathbb{C}$,
whereas for the second equation in \eqref{BLpower}, we need $n\in\mathbb{Z}$. For $n\in\mathbb{C}\setminus\mathbb{Z}$,
see \eqref{BLpowerGeneral}.

In this paper, we denote
\begin{equation}\label{X}
\begin{split}
    &\K{t}=\Cq \Eq{t}\quad{\rm if}\quad \ee=E,\\
    &\K{t}=\frac{-1}{\ln(q)\Thetaq{t}}\quad{\rm if}\quad \ee=\theta.
\end{split}
\end{equation}
Thus, $\int_0^\infty t^{n-1}\K{\frac{t}{z}}\intd t=q^{-\binom{n}{2}}z^n$.

In the proofs of the main theorems, we will use the following identities.
\begin{lemma}\label{lemmaparfrac}
For $x,y,x/z,x/y\notin \left\{0,-1,-q^{\pm1},-q^{\pm2},\ldots \right\}$, we have
\begin{align} 
    \frac{1}{\Thetaq{x}}&=\frac{1}{\qpr{q}{q}{\infty}^3}
    \sum_{n=-\infty}^\infty 
    \frac{\left(-1\right)^nq^{\binom{n}{2}}}{x+q^{-n}},\label{PartFrac1}\\
    \frac{\Thetaq{ax}}{\Thetaq{x}}&=
    \frac{\Thetaq{-a}}{\qpr{q}{q}{\infty}^3}
    \sum_{n=-\infty}^\infty 
    \frac{a^n}{1+x q^{n}},\qquad\qquad q<|a|<1,\label{PartFrac2}\\
    \frac{\Thetaq{ax}}{\Thetaq{x}}&=
    \frac{-\pi\Thetaq{-a}}{\qpr{q}{q}{\infty}^3\ln q}
    \sum_{n=-\infty}^\infty 
    \frac{x^{-(2n\pi\iunit+\ln a)/\ln q}}{\sin\left(\frac\pi{\ln q}(2n\pi\iunit+\ln a)\right)},\label{PartFrac2a}\\
    \frac{x\dThetaq{x}}{\Thetaq{x}}-\frac{y\dThetaq{y}}{\Thetaq{y}}
    &=\sum_{n=-\infty}^\infty \left(\frac{x}{x+q^n}-\frac{y}{y+q^n}\right),\label{lemmaparfrac1}\\
    \left(\frac{x\dThetaq{x}}{\Thetaq{x}}\right)'
    &=\sum_{n=-\infty}^\infty \frac{q^n}{\left(x+q^n\right)^2},\label{lemmaparfrac2}
\end{align}
\begin{equation}\label{lemmaparfrac3}
\begin{split}
    &\frac{x\dThetaq{x}}{\Thetaq{x}}
    -\frac{x\dThetaq{x/z}}{z\Thetaq{x/z}}
    -\frac{y\dThetaq{y}}{\Thetaq{y}}
    +\frac{y\dThetaq{y/z}}{z\Thetaq{y/z}}\\
    &\qquad\qquad\qquad\qquad\qquad\qquad=\frac{\qpr{q}{q}{\infty}^3\Thetaq{-z}\Thetaq{\frac{-x}{y}}
    \Thetaq{\frac{-xy}{z}}}{\Thetaq{x}\Thetaq{y}
    \Thetaq{\frac{x}{z}}\Thetaq{\frac{z}{y}}}.
\end{split}\end{equation}
\end{lemma}
\begin{proof}
Expansion \eqref{PartFrac1} can be obtained by considering the integral 
$\int_0^\infty\frac{\intd z}{(z-x)\Thetaq{z}}$ and rotating the contour of integration by $2\pi$.
Equation \eqref{PartFrac2} follows from \cite[Eq.(10.6.1)]{AAR1999}.
For \eqref{PartFrac2a}, we start with the constraints $x,a\in(q,1)$ and consider the residues of the
Mellin--Barnes type integral
\begin{equation}\label{tempMB}
\begin{split}
    &\frac{-\pi}{2\pi\iunit}\left(\int_{M+\frac12-\iunit\infty}^{M+\frac12+\iunit\infty}-
    \int_{-N-\frac12-\iunit\infty}^{-N-\frac12+\iunit\infty}\right)
    \frac{x^\tau}{\left(1-aq^\tau\right)\sin\pi\tau}\intd\tau\\
    &\qquad\qquad=
    \frac{-\pi}{\ln q}\sum_{n=-\infty}^\infty 
    \frac{x^{-(2n\pi\iunit+\ln a)/\ln q}}{\sin\left(\frac\pi{\ln q}(2n\pi\iunit+\ln a)\right)}
    -\sum_{m=-N}^M \frac{\left(-x\right)^m}{1-aq^m}.
\end{split}\end{equation}
We push $M,N\to\infty$. The left-hand side of \eqref{tempMB} vanishes, and the final term of \eqref{tempMB}
can be expressed in the required quotient of theta functions via \eqref{PartFrac2}.

Let $N$ be a positive integer and define $\Thetaq{\tau;N}=\qpr{q,-\tau,-q/\tau}{q}{N}$. 
The reader can check that
\begin{equation*}
    \frac{x\dThetaq{x;N}}{\Thetaq{x;N}}-\frac{y\dThetaq{y;N}}{\Thetaq{y;N}}
    =\sum_{n=1-N}^N \left(\frac{x}{x+q^n}-\frac{y}{y+q^n}\right).
\end{equation*}
Now let $N\to\infty$. For \eqref{lemmaparfrac2}, consider $\lim_{y\to x}\frac{\eqref{lemmaparfrac1}}{x-y}$.

We take $z=\expe^{2\iunit \zeta}$. 
As functions of $\zeta$, both expressions of \eqref{lemmaparfrac3} are elliptic
functions of order 2 with periods $\pi$ and $\frac{\ln q}{2\iunit}$. It is easy to show that both expressions 
of \eqref{lemmaparfrac3} have the same poles, zeros, and the same residues at the poles. 
Hence, the quotient of both sides is constant, and because of the same residues, this constant is unity.
\end{proof}

An additional identity obtained for the ratio of $\Eq{\tau}$ functions is
\begin{equation} \label{Eqab}
\frac{\Eq{at}\Eq{bt}}{\Eq{t}\Eq{abt}}=\expe^{-\ln(a)\ln(b)/\ln(q)}.
\end{equation}

A function $f$ with the property $f(q\tau)=f(\tau)$ is called $q$-periodic.
We define the $q$-periodic function (recall $\qhat=\expe^{\frac{2\pi^2}{\ln q}}$)
\begin{equation}\label{qperiodic}
\Pq{\tau}=-\ln(q)\Cq\Eq{\tau}\Thetaq{\tau}=
\qpr{\qhat^2}{\qhat^2}{\infty}
    \exp\left(\sum_{n=1}^\infty\frac{\cos\left(2\pi n\frac{\ln\tau}{\ln q}\right)}{n%
    \sinh\left(n\ln \qhat\right)}
    \right).
\end{equation}
This expansion is an alternative version of \cite[(3.7)]{Joshi2025}.
We also have
\begin{equation} \label{Pq}
    \Pq{q^t}=-\ln(q)\Cq \sum_{n=-\infty}^\infty q^{\frac12\left(t-\frac12+n\right)^2},
\end{equation}
with Fourier series,
\begin{equation} \label{PqFourier}
    \Pq{q^t}=1+2\sum_{n=1}^\infty \left(-1\right)^n\qhat^{n^2} \cos(2n\pi t)=\sum_{n=-\infty}^\infty (-1)^n\qhat^{n^2}\expe^{2n\pi\iunit t}=\theta_4(\pi t,\qhat).
\end{equation}
This clearly demonstrates that $\Pq{\tau}\approx 1$ for $q$ bounded away from $0$.
It follows from Fourier series \eqref{PqFourier} that
\begin{equation} \label{PqInt}
    \frac{-1}{\ln q}\int_q^1 \Pq{z\tau}\frac{\intd\tau}\tau=\int_0^1\Pq{zq^t}\intd t
    =\int_0^1\Pq{q^{t+\frac{\ln z}{\ln q}}}\intd t=\int_0^1\Pq{q^t}\intd t=1.
\end{equation}
We may also derive, from \cite[Example 14, p.\ 489]{WW1927},
\begin{equation}\label{PqReciprocal}
\frac1{\Pq{q^t}}=\frac1{\qpr{\qhat^2}{\qhat^2}{\infty}^3}
\left(\widetilde{a}_0+2\sum_{n=1}^\infty \widetilde{a}_n\qhat^n\cos(2n\pi t)\right),
\end{equation}
with
\begin{equation}\label{tan}
    \widetilde{a}_n=\sum_{m=0}^\infty \left(-1\right)^m
    \qhat^{m(m+2n+1)}
    =\qphi{1}{1}{\qhat^2}{0}{\qhat^2}{\qhat^{2n+2}}.
\end{equation}
These coefficients have the following recurrence and normalising relations,
\begin{equation}
    \qhat^{2n+2}\widetilde{a}_{n+1}+\widetilde{a}_n=1,\qquad 
    \sum_{n=0}^\infty \frac{\left(-1\right)^n\qhat^{n(n+1)}\widetilde{a}_n}{
    \qpr{\qhat^2}{\qhat^2}{n}}=\qpr{\qhat^2}{\qhat^2}{\infty}^2.
\end{equation}
The function $\Eq{\tau}$ and hence $\Pq{\tau}$ are multi-valued:
\begin{equation}
    \frac{\Eq{\tau\expe^{2\pi\iunit}}}{\Eq{\tau}}=\frac{\Pq{\tau\expe^{2\pi\iunit}}}{\Pq{\tau}}=
    -\qhat^{-1}\tau^{\frac{2\pi\iunit}{\ln q}}.
\end{equation}
Using the translation of the argument by a half-period, see
\cite[\href{http://dlmf.nist.gov/20.2.E10}{20.2.10}]{NIST:DLMF}, we obtain
\begin{equation}\label{PqMinus}
\Pq{\expe^{\pi\iunit}q^t}
=-\iunit\expe^{\pi\iunit t}\qhat^{-\frac14}\theta_1(\pi t,\qhat).
\end{equation}
The function $\Pq{t}$ is the bridge between the two continuous $q$-Laplace transforms:
\begin{equation}
\begin{split}
    (\LqK{E} B)(z)=\Cq\int_0^\infty B(t) \Eq{\frac{t}{z}}\, \frac{\intd t}{t}&=
    \frac{-1}{\ln q}\int_0^\infty \frac{B(t)}{\Thetaq{\frac{t}{z}}} 
    \Pq{\frac{t}{z}}\, \frac{\intd t}{t},\\
    (\LqK{\theta} B)(z)=\frac{-1}{\ln q}\int_0^\infty \frac{B(t)}{\Thetaq{\frac{t}{z}}} 
    \, \frac{\intd t}{t}&=
    \Cq\int_0^\infty \frac{B(t) \Eq{\frac{t}{z}}}{\Pq{\frac{t}{z}}}\, \frac{\intd t}{t}.
\end{split}
\end{equation}
Finally, we observe that a quotient of two $\theta_q$ functions is a simple power times a
$q$-periodic function.
\begin{equation}
    \frac{\Thetaq{ax}}{\Thetaq{bx}}=x^{\frac{\ln(b/a)}{\ln q}}\frac{\Eq{b}\Pq{ax}}{\Eq{a}\Pq{bx}}.
\end{equation}


\section{$q$-Laplace Transforms}\label{S:Laplace}
For a series $B(t)=\sum_{n=0}^\infty a_nt^n$, the $q$-Laplace transform is formally defined as 
$(\Lq B)(z)=\sum_{n=0}^\infty a_nq^{-\binom{n}{2}}z^n$, that is, the three $q$-Laplace transforms give
rise to the same divergent asymptotic expansions, but with different sectors of validity. 
Hence, the difference between two of these
$q$-Laplace transforms should be $q$-exponentially small. The $q$-Laplace transform $(\LqK{\lambda} B)(z)$
is typically one-valued, whereas $(\LqK{E} B)(z)$ and $(\LqK{\theta} B)(z)$ are typically multi-valued. 
Let us assume the simplest case in which they have connection formulas of the form
\begin{equation}\label{connectionsimple}
\begin{split}
    (\LqK{E} B)(z\expe^{2\pi\iunit})&=(\LqK{E} B)(z)-2\pi\iunit\Cq S(z)\Eq{z\expe^{\pi\iunit}},\\
    (\LqK{\theta} B)(z\expe^{2\pi\iunit})&=(\LqK{\theta} B)(z)+\frac{2\pi\iunit S(z)}{\ln (q)\Thetaq{-z}},
\end{split}
\end{equation}
with $S(z)$ an entire function. The final terms in \eqref{connectionsimple} are the $q$-exponentially small
terms that will be switched on via the Stokes phenomenon.
The growth of $\Eq{z}$ lies in between algebraic and exponential as $z\to 0$ and $z\to \infty$.
The function $1/\Thetaq{z}$ has the same growth rate as long as we stay away from the poles.
One big difference with normal exponentials is that $\Eq{z}$ is $q$-exponentially small on all Riemann
sheets. Hence, in the case of $(\LqK{E} B)(z)$, the Stokes phenomenon switches on terms that will be
$q$-exponentially small on all Riemann sheets. For $(\LqK{\theta} B)(z)$, the switched-on
terms will be $q$-exponentially small until we encounter the poles of $1/\Thetaq{-z}$.
These observations result in the following sectors of validity:
\begin{equation}
    \begin{split}
        (\LqK{E} B)(z)&\sim \sum_{n=0}^\infty a_nq^{-\binom{n}{2}}z^n,\qquad
        {\rm for~all~}\arg z\\
        (\LqK{\theta} B)(z)&\sim \sum_{n=0}^\infty a_nq^{-\binom{n}{2}}z^n,\qquad
        |\arg z|<2\pi,\\
        (\LqK{\lambda} B)(z)&\sim \sum_{n=0}^\infty a_nq^{-\binom{n}{2}}z^n,\qquad
        |\arg (z/\lambda)|<\pi,
    \end{split}
\end{equation}
as $|z|\to 0$. In Section \ref{S:errorbounds}, we will illustrate this for the case of $\qphi{2}{0}{a,b}{-}{q}{z}$.

In the following theorem, we give an exact formula for the difference of two $q$-Laplace transforms.
The function $S(z)$ that is switched on via the Stokes phenomenon plays a crucial role.
The proof is long and is moved to the Appendix.

\begin{theorem}\label{thmCompareqL}
Let $S(t)$ be an entire function of order and type zero with the additional constraint that
there exist $M>0$ and $c\in(0,1)$ such that $|S(t)|\leq M\qpr{-c|t|}{q}{\infty}$ for all $t\in\mathbb{C}$.
If we have the Cauchy-Heine integral representation
\begin{equation}
    (\LqK{\theta} B)(z)=\frac{-1}{\ln q}\int_0^\infty \frac{S(-t)}{\Thetaq{t}(t+z)} 
    \, \intd t,
\end{equation}
or equivalently, if connection formulas \eqref{connectionsimple} hold,
then
\begin{align}
    (\LqK{\theta} B)(z)-(\LqK{E} B)(z)&=\frac{S(z)}{\Thetaq{-z}}P_q^{(c)}(z),\label{Compare1}\\
    (\LqK{\theta} B)(z)-(\LqK{\lambda} B)(z)&=\frac{S(z)}{\Thetaq{-z}}P_q^{(d)}(z;\lambda)\label{Compare2},
\end{align}
in which
\begin{align}
    P_q^{(d)}(z;\lambda)&=\frac{\lambda\dThetaq{\lambda}}{\Thetaq{\lambda}}
    -\frac{\lambda\dThetaq{\lambda/z}}{z\Thetaq{\lambda/z}}+\frac{\ln z}{\ln q}
    =\frac{\lambda\dPq{\lambda}}{\Pq{\lambda}}-\frac{\lambda\dPq{\frac{\lambda}{z}}}{z\Pq{\frac{\lambda}{z}}},\label{Pqd1}\\
    P_q^{(c)}(z)&=\frac{-1}{\ln q}\int_{q}^1 \Pq{\frac{t}{z}}
    P_q^{(d)}(z;t)\frac{\intd t}{t}\label{Pqc0}\\
    &=\frac{\ln z}{\ln q}+\frac{1}{z\ln q}\int_{q}^1 \Pq{t}\frac{\dThetaq{t/z}}{\Thetaq{t/z}}
    \intd t=\frac{1}{z\ln q}\int_q^1\frac{\Pq{t}\dPq{\frac{t}{z}}}{\Pq{\frac{t}{z}}}\intd t\label{PqcPq}\\
    &=\frac{-2\pi}{\ln q}\sum_{n=1}^\infty \left(-1\right)^n\qhat^{n^2}
    \frac{\sin\left(2\pi n\frac{\ln z}{\ln q}\right)}{\sinh\left(n\ln \qhat\right)}\label{Pqc1},
\end{align}
\begin{align}
     P_q^{(d)}(z;\lambda_1)-P_q^{(d)}(z;\lambda_2)&=\frac{\lambda_1\dThetaq{\lambda_1}}{\Thetaq{\lambda_1}}
    -\frac{\lambda_1\dThetaq{\lambda_1/z}}{z\Thetaq{\lambda_1/z}}
    -\frac{\lambda_2\dThetaq{\lambda_2}}{\Thetaq{\lambda_2}}
    +\frac{\lambda_2\dThetaq{\lambda_2/z}}{z\Thetaq{\lambda_2/z}}\nonumber\\
    &=\frac{\qpr{q}{q}{\infty}^3\Thetaq{-z}\Thetaq{\frac{-\lambda_1}{\lambda_2}}
    \Thetaq{\frac{-\lambda_1\lambda_2}{z}}}{\Thetaq{\lambda_1}\Thetaq{\lambda_2}
    \Thetaq{\frac{\lambda_1}{z}}\Thetaq{\frac{z}{\lambda_2}}},\label{Pqd2}
\end{align}
are $q$-periodic functions. The right-hand sides of \eqref{Compare1} and \eqref{Compare2} have removable
singularities at $z=q^m$, $m\in\mathbb{Z}$:
\begin{align}
    \lim_{z\to q^m}\frac{P_q^{(c)}(z)}{\Thetaq{-z}}&=
    \frac{\left(-1\right)^m q^{\binom{m}{2}}}{\qpr{q}{q}{\infty}^3}\left(\frac{2\pi}{\ln q}\right)^2
    \sum_{n=1}^\infty \left(-1\right)^n\qhat^{n^2}
    \frac{n}{\sinh\left(n\ln \qhat\right)},\label{Compare1a}\\
    \lim_{z\to q^m}\frac{P_q^{(d)}(z;\lambda)}{\Thetaq{-z}}&=
    \frac{\left(-1\right)^{m-1} q^{\binom{m}{2}}}{\qpr{q}{q}{\infty}^3}
    \left(\frac1{\ln q}+\sum_{n=-\infty}^\infty\frac{\lambda q^n}{\left(1+\lambda q^n\right)^2}
    \right)\label{Compare2a}\\
    &=\frac{\left(-1\right)^{m-1} q^{\binom{m}{2}}}{\qpr{q}{q}{\infty}^3}
    \left(\frac1{\ln q}+\lambda\frac{\intd}{\intd \lambda}\left(\frac{\lambda\dThetaq{\lambda}}{\Thetaq{\lambda}}
    \right)\right)\nonumber.
\end{align}
\end{theorem}

Note that the periodic functions $P_q^{(c)}(z)$ and $P_q^{(d)}(z;\lambda)$ do not depend on $B(t)$
and $S(z)$. In \cite{DiVizio2009} the $\LqK{\theta}-\LqK{\lambda}$ difference of the resummation 
of the $q$-Euler series is discussed. Their result is equivalent to \eqref{Compare2} with $S(z)=1$
and the final representation for $P_q^{(d)}(z;\lambda)$ in \eqref{Pqd1}. We note that our
presentation for $P_q^{(d)}(z;\lambda)$ seems much simpler.

The requirement that $S(z)$ is an entire function might seem strong. In a sequel to this
paper, we will demonstrate that the same universal functions $P_q^{(c)}(z)$ and $P_q^{(d)}(z;\lambda)$
will show up when we discuss the asymptotics of solutions of three term $q$-difference equations.
These periodic functions have simple connection formulas:

\begin{corollary}\label{crllPqcdStokes}
The $P_q^{(c)}(z)$ and $P_q^{(d)}(z;\lambda)$ functions have the connection formulas
\begin{equation}\label{PqcStokes}
    P_q^{(c)}(z\expe^{2\pi\iunit})-P_q^{(c)}(z)=-\frac{2\pi\iunit}{\ln q}\left(\Pq{z\expe^{\pi\iunit}}-1\right),
\end{equation}
\begin{equation}\label{PqdStokes}
    P_q^{(d)}(z\expe^{2\pi\iunit};\lambda)-P_q^{(d)}(z;\lambda)=\frac{2\pi\iunit}{\ln q}.
\end{equation}
\end{corollary}


\section{Application: the simplest divergent basic hypergeometric function}\label{S:Application}

The simplest divergent basic hypergeometric $_r\phi_s$ series is the $_2\phi_0$ series. 
The formal $_2\phi_0$ series will have three $q$-Borel-Laplace transforms denoted by
$\UqK{\ee}{a}{b}{z}$, in which $\ee$ is either $E$, $\theta$ or $\lambda$.
More explicitly,
\begin{align}
    \UqK{\ee}{a}{b}{z}&= \int_0^{\infty}
    \qphi{2}{1}{a,b}{0}{q}{-t}\K{\frac{t}{z}}\frac{\intd t}{t},\label{2phi0BorelLaplaceK}\\
    \UqK{\lambda}{a}{b}{z}&=\sum_{n=-\infty}^\infty 
    \frac{\qphi{2}{1}{a,b}{0}{q}{-q^n\lambda}}{\Thetaq{\frac{q^n\lambda}{z}}},\label{2phi0BorelLaplaceL}
\end{align}
with the notation in \eqref{X}.
We do use the ${\rm U}$ notation to indicate that this
function is a $q$-version of the Kummer $U$ function which has the representation
$U(a,b;z)=z^{-a}{}_2F_0\left(\mytop{a,a-b+1}{-};-z^{-1}\right)$, see
\cite[\href{http://dlmf.nist.gov/13.6.E21}{13.6.21}]{NIST:DLMF}.
Hence, the $\UqK{\ee}{a}{b}{z}$ should represent an important class of special functions.

The $q$-difference equation satisfied by $\qphi{2}{0}{a,b}{-}{q}{z}$ and $\UqK{\ee}{a}{b}{z}$ is \eqref{2phi0DiffEq}.
An independent second solution to \eqref{2phi0DiffEq} near $z=0$ is 
\begin{equation}\label{second}
\begin{split}
    y_2(z)&=\K{-qz}\qpr{abz}{q}{\infty}\qphi{2}{1}{\frac{q}{a},\frac{q}{b}}{0}{q}{abz}\\
    &=\K{-qz}\qpr{bqz}{q}{\infty}\qphi{1}{1}{\frac{q}{a}}{bqz}{q}{aqz},
\end{split}\end{equation}
with Wronskian relation
\begin{equation}\label{Wronskian}
    \UqK{\ee}{a}{b}{z}y_2(z/q)-\UqK{\ee}{a}{b}{z/q}y_2(z)=\K{-z}\qpr{abz}{q}{\infty}.
\end{equation}
If $\ee=E$ or $\ee=\theta$ then we use the same $\kappa$ on both sides according to \eqref{X},
and if $\ee=\lambda$ then we use the same $\kappa$ in \eqref{second} and on the right-hand side of
\eqref{Wronskian}.

Two independent solutions near $z=\infty$ are
\begin{equation}
\begin{split}
    y_3(z)&=z^{-\ln(a)/\ln(q)}\qphi{2}{1}{a,0}{\frac{aq}{b}}{q}{\frac{q}{abz}},\\
    y_4(z)&=z^{-\ln(b)/\ln(q)}\qphi{2}{1}{b,0}{\frac{bq}{a}}{q}{\frac{q}{abz}}.
\end{split}\end{equation}
In \cite{Morita2013}, the connection relations between the solutions of \eqref{2phi0DiffEq} are
discussed, focusing on the relation between $y_2(z)$, $y_3(z)$ and $y_4(z)$.
The relation between $\UqK{\lambda}{a}{b}{z}$, $y_3(z)$ and $y_4(z)$ is mentioned.
In Section \ref{S:connections}, we express $\UqK{E}{a}{b}{z}$ and $\UqK{\theta}{a}{b}{z}$
in terms of $y_3(z)$ and $y_4(z)$.

Hence, we have three $q$-Borel-Laplace resummations for $\qphi{2}{0}{a,b}{-}{q}{z}$. It follows from
Theorem \ref{thmCompareqL} that
\begin{equation}
\begin{split}
    \UqK{\theta}{a}{b}{z}-\UqK{E}{a}{b}{z}&=\frac{\qpr{a,b,abz}{q}{\infty}}{\qpr{q}{q}{\infty}}
    \qphi{2}{1}{\frac{q}{a},\frac{q}{b}}{0}{q}{abz}\frac{P_q^{(c)}(z)}{\Thetaq{-qz}},\\
    \UqK{\theta}{a}{b}{z}-\UqK{\lambda}{a}{b}{z}&=\frac{\qpr{a,b,abz}{q}{\infty}}{\qpr{q}{q}{\infty}}
    \qphi{2}{1}{\frac{q}{a},\frac{q}{b}}{0}{q}{abz}\frac{P_q^{(d)}(z;\lambda)}{\Thetaq{-qz}},
\end{split}
\end{equation}
that is, the difference is the second solution (see \eqref{second}) times a universal $q$-periodic function.


\section{Integral and Sum Representations}\label{S:integrals}
The Kummer $U$ is an important special function, with many well-known special cases. It has
several important and useful integral representations. In this section, we will derive similar representations
for $\UqK{\ee}{a}{b}{z}$. Again, using \eqref{X}, and 
$K_0=\frac{\qpr{a,b}{q}{\infty}}{\qpr{q}{q}{\infty}}$.
\begin{theorem}\label{thmIntegrals}
    \begin{align}
        \UqK{\ee}{a}{b}{z}&=\int_0^\infty \frac{\qpr{-bt}{q}{\infty}}{\qpr{-t}{q}{\infty}}
        \qphi{1}{1}{b}{-bt}{q}{-at}\K{\frac{t}{z}} \frac{\intd t}{t},\label{Uq2phi1}\\
        &=\qpr{a}{q}{\infty}\int_0^\infty\frac{\qpr{-bt,-\frac{aqz}{t}}{q}{\infty}}{\qpr{-t}{q}{\infty}}
        \K{\frac{t}{z}}\frac{\intd t}{t}, \quad |aq|<1,\label{UqPoch}\\
        &=\qpr{abz}{q}{\infty}\int_0^\infty
        \frac{\qpr{-at,-bt}{q}{\infty}}{\qpr{-t}{q}{\infty}}
        \K{\frac{t}{z}} \frac{\intd t}{t}, \quad |abz|<1,\label{UqSymmetric}
    \end{align}
    and
\begin{equation} \label{UqCauchyHeine}
    \UqK{\ee}{a}{b}{z}=K_0\int_0^\infty\frac{\qpr{-bqt}{q}{\infty}\qphi{1}{1}{\frac{q}{a}}{-bqt}{q}{-aqt}}{t+z}\K{qt}\intd t,
\end{equation}
when $|a|,|b|<1$.
\end{theorem}

\begin{theorem}\label{thmSums}
    \begin{align}
    \UqK{\lambda}{a}{b}{z}&
    =\sum_{n=-\infty}^\infty \frac{\qpr{-bq^n\lambda}{q}{\infty}\qphi{1}{1}{b}{-bq^n\lambda}{q}{
    -aq^n\lambda}}{\qpr{-q^n\lambda}{q}{\infty}\Thetaq{\frac{q^n\lambda}{z}}}\label{Lsum2}\\
    &=\qpr{a}{q}{\infty}\sum_{n=-\infty}^\infty\frac{\qpr{-bq^n\lambda,-\frac{aqz}{q^n\lambda}}{q}{\infty}}{\qpr{-q^n\lambda}{q}{\infty}\Thetaq{\frac{q^n\lambda}{z}}}, \qquad |aq|<1,\label{Lsum3}\\
    &=\qpr{abz}{q}{\infty}\sum_{n=-\infty}^\infty
    \frac{\qpr{-aq^n\lambda,-bq^n\lambda}{q}{\infty}}{
    \qpr{-q^n\lambda}{q}{\infty}\Thetaq{\frac{q^n\lambda}{z}}}, \quad |abz|<1,\label{Lsum4}
\end{align}
and
\begin{equation}\label{Lsum5}
    \UqK{\lambda}{a}{b}{z}=K_0\sum_{n=-\infty}^\infty\frac{\qpr{-bq^{n+1}\lambda}{q}{\infty}\qphi{1}{1}{\frac{q}{a}}{-bq^{n+1}\lambda}{q}{-aq^{n+1}\lambda}}{\Thetaq{q^{n+1}\lambda}}\frac{q^n\lambda}{q^n\lambda+z}.
\end{equation}
\end{theorem}
Mellin--Barnes integral representations:
\begin{theorem}\label{thmMB}
    \begin{align}
    \UqK{E}{a}{b}{z}&
    =\frac{-K_0}{2\pi\iunit}\int_{-\iunit\infty}^{\iunit\infty}
    \frac{\pi\qpr{q^{1+s}}{q}{\infty}q^{-\binom{s}{2}}z^s}{\sin(\pi s)\qpr{aq^s,bq^s}{q}{\infty}}
    \intd s,\label{MB1}\\
    \UqK{\theta}{a}{b}{z}&
    =\frac{K_0}{2\pi\iunit\qpr{q}{q}{\infty}^3\ln q}\int_{-\iunit\infty}^{\iunit\infty}
    \frac{\pi^2\qpr{q^{1+s}}{q}{\infty}\Thetaq{-q^{s}}z^s}{\sin^2(\pi s)\qpr{aq^s,bq^s}{q}{\infty}}
    \intd s,\label{MB2}\\
    \UqK{\lambda}{a}{b}{z}&
    =\frac{-K_0}{2\pi\iunit}\int_{-\iunit\infty}^{\iunit\infty}
    \frac{\pi\qpr{q^{1+s}}{q}{\infty}\lambda^s\Thetaq{\frac{\lambda}{z}q^s}}{
    \sin(\pi s)\qpr{aq^s,bq^s}{q}{\infty}\Thetaq{\frac{\lambda}{z}}}
    \intd s,\label{MB3}
\end{align}
in which the contour of integration separates the poles at $s=0,1,2,\ldots$, from the poles at
$s+\frac{\ln a}{\ln q}=-m-n\frac{2\pi\iunit}{\ln q}$, $s+\frac{\ln b}{\ln q}=-m-n\frac{2\pi\iunit}{\ln q}$ with
$m\in\mathbb{N}_0$, and $n\in\mathbb{Z}$.
\end{theorem}

Integral/sum representations \eqref{Uq2phi1} and \eqref{Lsum2} are just alternative versions of the $q$-Borel-Laplace transforms. Representations \eqref{UqPoch} and \eqref{Lsum3} are in terms of more elementary functions.
Similarly for \eqref{UqSymmetric} and \eqref{Lsum4}, but note that these representations are also symmetric
in $a$ and $b$. Finally, the Cauchy-Heine representations \eqref{UqCauchyHeine} and \eqref{Lsum5} are very
convenient for remainder estimates in Section \ref{S:errorbounds}.

Integral representation \eqref{UqSymmetric} has the simplest integrand. When we consider its poles at
$t=\expe^{-\pi\iunit}q^{-n}$, $n=0,1,2,\ldots$, we obtain the connection formula:
\begin{equation}\label{connectionE}
\begin{split}
    \UqK{E}{a}{b}{z\expe^{2\pi\iunit}}&-\UqK{E}{a}{b}{z}\\
    &=-2\pi\iunit K_0\Cq\Eq{qz\expe^{\pi\iunit}}\qpr{abz}{q}{\infty}\qphi{2}{1}{\frac{q}{a},\frac{q}{b}}{0}{q}{abz}\\
    &=-2\pi\iunit K_0\Cq\Eq{qz\expe^{\pi\iunit}}\qpr{bqz}{q}{\infty}\qphi{1}{1}{\frac{q}{a}}{bqz}{q}{aqz}.
\end{split}
\end{equation}
For the $\K{\tau}=\frac{-1}{\ln(q)\Thetaq{\tau}}$ kernel, the connection formula is the same but with the $\Cq\Eq{qz\expe^{\pi\iunit}}$ replaced by $\frac{-1}{\ln(q)\Thetaq{-qz}}$.


\section{Connection Relations}\label{S:connections}

The case $\ee=\lambda$ of the following theorem is a special case of \cite[Theorem 3.1]{Adachi2019}.
\begin{theorem}\label{thmConnection}
For the choices $\ee=E,\theta,\lambda$ and $\frac{b}{a}\not\in \left\{q^n: n\in\mathbb{Z}\right\}$, we have
\begin{equation}\label{connectionK}
\begin{split}
    \UqK{\ee}{a}{b}{z}=&\pK{\ee}{a}{z}\left(az\right)^{-\frac{\ln a}{\ln q}}
    \frac{\qpr{b}{q}{\infty}}{\qpr{\frac{b}{a}}{q}{\infty}}
    \qphi{2}{1}{a,0}{\frac{aq}{b}}{q}{\frac{q}{abz}}\\
    &+\pK{\ee}{b}{z}\left(bz\right)^{-\frac{\ln b}{\ln q}}\frac{\qpr{a}{q}{\infty}}{\qpr{\frac{a}{b}}{q}{\infty}}
    \qphi{2}{1}{b,0}{\frac{bq}{a}}{q}{\frac{q}{abz}},
\end{split}
\end{equation}
in which the $\pK{\ee}{a}{z}$ are $q$-periodic in both variables. More explicitly,
\begin{equation}\label{pKlambda}
    \pK{\lambda}{a}{z}=\frac{\Pq{a\lambda}\Pq{\frac{\lambda}{az}}}{\Pq{\lambda}\Pq{\frac{\lambda}{z}}},
\end{equation}
\begin{equation}\label{pKE}
    \pK{E}{q^\alpha}{q^\zeta}=\frac{\sqrt{\frac{-2\pi^3}{\ln^3 q}}}{q^{\frac18}\qpr{q}{q}{\infty}^3}
    \theta_1\left(\pi\alpha,\qhat\right)
    \sum_{n=-\infty}^\infty(-1)^n\qhat^{n^2}\frac{
    \expe^{-2n\pi\iunit(\alpha+\zeta)}}{
    \sin(\pi\alpha+\iunit n\ln\hat{q})},
\end{equation}
for an alternative representation see \eqref{pKE2}, and
\begin{equation}\label{pKtheta1}
    \pK{\theta}{q^\alpha}{q^\zeta}=\frac{\sqrt{\frac{-2\pi^3}{\ln^3 q}}
    \theta_1\left(\pi\alpha,\qhat\right)\theta_1\left(\pi(\alpha+\zeta),\qhat\right)}{
    q^{\frac18}\qpr{q}{q}{\infty}^3\theta_1\left(\pi \zeta,\qhat\right)}\left(
    \frac{\theta'_1\left(\pi\alpha,\qhat\right)}{\theta_1\left(\pi \alpha,\qhat\right)}
    -\frac{\theta'_1\left(\pi(\alpha+\zeta),\qhat\right)}{\theta_1\left(\pi(\alpha+\zeta),\qhat\right)}
    \right).
\end{equation}
For alternative representations see \eqref{pKtheta2} and \eqref{pKtheta3}.
\end{theorem}
We note that for most reasonable choices of the parameters, we have for the $q$-periodic multipliers that
$\pK{\ee}{a}{z}\approx 1$. Note also that $\pK{\ee}{q^n}{z}=1$, $n\in\mathbb{Z}$, and
$\UqK{\ee}{q^{-n}}{b}{z}=z^nq^{-n^2}\qpr{b}{q}{n}\qphi{2}{1}{q^{-n},0}{q^{1-n}/b}{q}{\frac{q^{1+n}}{bz}}$,
$n\in\mathbb{N}_0$. Compare \cite[(1.13.16)]{KLS2010}.

Taking the $b\to a q^m$ limit in \eqref{connectionK}, we obtain the following result.
\begin{corollary}
    When $b=aq^m$, $m=1,2,3,\ldots$, we have for all three choices $\ee=E,\theta,\lambda$ that
    \begin{equation}\begin{split}\label{connectionKaa}
    \left(az\right)^{-\frac{\ln a}{\ln q}}
    &\frac{\qpr{q}{q}{\infty}}{\qpr{a}{q}{\infty}}\UqK{\ee}{a}{b}{z}\\
    =&-\pK{\ee}{a}{z}\frac{\qpr{q}{q}{m}}{q^m\qpr{a}{q}{m}}
    \sum_{n=0}^{m-1}\frac{\qpr{a}{q}{n}}{\qpr{q^{-m}}{q}{n+1}\qpr{q}{q}{n}}
    \left(\frac{q}{abz}\right)^n\\
    &+\left(\pK{\ee}{a}{z}\left(1-\frac{\ln\left(abz\right)}{\ln q}\right)+a\frac{\partial \pK{\ee}{a}{z}}{\partial a}
    \right)\frac{q^{\binom{m}{2}}\qphi{2}{1}{b,0}{q^{m+1}}{q}{\frac{q}{abz}}}{\qpr{q}{q}{m}\left(-abz/q\right)^{m}}\\
    &+\pK{\ee}{a}{z}\left(-1\right)^m q^{\binom{m}{2}}
    \sum_{n=0}^\infty\frac{\qpr{b}{q}{n}}{\qpr{q}{q}{m+n}\qpr{q}{q}{n}}\left(\frac{q}{abz}\right)^{m+n}\\
    &\qquad\qquad\qquad\qquad\qquad
    \times\left(\Psiq{bq^{n}}-\Psiq{q^{m+n+1}}-\Psiq{q^{n+1}}
    \right),
    \end{split}\end{equation}
    with
    \begin{equation}
    \Psiq{a}=\sum_{\ell=0}^\infty\frac{aq^\ell}{1-aq^\ell}.
\end{equation}
\end{corollary}
In the case $m=0$, that is $b=a$, the first term on the right-hand side of \eqref{connectionKaa} vanishes.


\section{Error Bounds}\label{S:errorbounds}
The three functions $\UqK{\ee}{a}{b}{z}$ are resummations of $\qphi{2}{0}{a,b}{-}{q}{z}$. We still have
to show that this divergent series is the asymptotic expansion of $\UqK{\ee}{a}{b}{z}$ as $z\to0$.
We denote \eqref{UqCauchyHeine} as
\begin{equation} \label{Bound0}
    \UqK{\ee}{a}{b}{z}=K_0\int_0^\infty\frac{u(t)}{t+z}\K{qt}\intd t,
\end{equation}
with $K_0=\frac{\qpr{a,b}{q}{\infty}}{\qpr{q}{q}{\infty}}$ and
\begin{equation}\label{ut}
    \begin{split}
        u(t)&=\qpr{-abt}{q}{\infty}\qphi{2}{1}{\frac{q}{a},\frac{q}{b}}{0}{q}{-abt}
        =\qpr{-aqt}{q}{\infty}\qphi{1}{1}{\frac{q}{b}}{-aqt}{q}{-bqt}\\
        &=\frac{\qpr{\frac{q}{a},-aqt,\frac{-1}{at}}{q}{\infty}}{
        \qpr{\frac{b}{a}}{q}{\infty}}\qphi{2}{1}{a,0}{\frac{aq}{b}}{q}{\frac{-q}{abt}}
         +\frac{\qpr{\frac{q}{b},-bqt,\frac{-1}{bt}}{q}{\infty}}{
        \qpr{\frac{a}{b}}{q}{\infty}}\qphi{2}{1}{b,0}{\frac{bq}{a}}{q}{\frac{-q}{abt}}.
    \end{split}
\end{equation}
The first line in \eqref{ut} shows that $u(t)$ is bounded for finite $t$, and the final
representation shows that for large $t$, the growth of $u(t)$ is dominated by the factors
$\qpr{-aqt}{q}{\infty}$ and $\qpr{-bqt}{q}{\infty}$. Let $c=\max(|a|,|b|)$ and for $M_q(a,b)$, a positive
constant such that
\begin{equation}\label{Mqab}
    |u(t)|\leq M_q(a,b) \qpr{-cq|t|}{q}{\infty},\quad {\rm for~all~}t{\rm ~in~the~half~plane}\quad \Re t\geq0.
\end{equation}

\begin{theorem}\label{thmErrorbounds}
We take $c=\max(|a|,|b|)$ and for $M_q(a,b)$, a positive
constant such that \eqref{Mqab} holds.
Take $N$ a positive integer with $N>1-\frac{\ln c}{\ln q}$ and define the remainder via
\begin{equation} \label{Bound1}
    \UqK{\ee}{a}{b}{z}=\sum_{n=0}^{N-1}\frac{\qpr{a,b}{q}{n}}{\qpr{q}{q}{n}}q^{-\binom{n}{2}}
    \left(-z\right)^n+R_{N,\ee}(a,b;z).
\end{equation}
Then in the half-plane $\Re z\geq 0$, we have
\begin{equation} \label{Bound2}
    \left|R_{N,\ee}(a,b;z)\right|\leq\left|z^N\right| \left|K_0\right|
    \frac{M_q(a,b)q^{-\binom{N}{2}}}{\qpr{cq^N}{q}{\infty}},
\end{equation}
for all three cases $\ee=E,\theta,\lambda$. In the case $\ee=\lambda$, this bound holds for 
$\lambda> 0$, and we do not require $N>1-\frac{\ln c}{\ln q}$.

In the sectors $\left|\arg z\right|\in(\frac\pi2,\pi)$, we have
\begin{align} 
    \left|R_{N,E}(a,b;z)\right|&\leq\left|z^N\right| \left|K_0\right|
    \frac{M_q(a,b)\exp\left({\frac{-\zeta^2}{2\ln q}}\right)q^{-\binom{N}{2}}}{\qpr{cq^N}{q}{\infty}},
    \label{Bound3}\\
    \left|R_{N,\theta}(a,b;z)\right|&\leq\left|z^N\right| \left|K_0\right|
    \frac{M_q(a,b)\exp\left({\frac{-\zeta^2}{2\ln q}}\right)q^{-\binom{N}{2}}}{\qpr{cq^N}{q}{\infty}\qpr{\qhat^2}{\qhat^2}{\infty}^3}\left(\frac{1+\sqrt{\qhat}}{1-\sqrt{\qhat}}\right),
    \label{Bound4}
\end{align}
with $\zeta=\left|\arg z\right|-\frac\pi2$ and
\begin{equation}
    \left|R_{N,\lambda}(a,b;z)\right|\leq\left|z^N\right| \left|K_0\right|
    \frac{M_q(a,b)q^{-\binom{N}{2}}}{\qpr{cq^N}{q}{\infty}|\sin\xi|},
    \label{Bound5b}
\end{equation}
where $\xi=\arg z$.
\end{theorem}

In the half-plane $\Re z\geq 0$, we have the relatively sharp error bound \eqref{Bound2},
and the bound is the same for the three cases! In the case $\ee=E$, the Stokes phenomenon will take place
at $\arg z=\pm\pi$, that is, the $q$-exponentially small term that appears on the right-hand side of
\eqref{connectionE} will be switched on when we cross $\arg z=\pm\pi$. As a consequence, in the sectors 
$\left|\arg z\right|\in(\frac\pi2,\pi)$, the error bound
increases, in this case, by a factor $\exp\left({\frac{-\zeta^2}{2\ln q}}\right)$ in \eqref{Bound3}. 
This is very similar for the error bounds for the Kummer $U$ function: see 
\cite[\href{https://dlmf.nist.gov/13.7.ii}{\S 13.7.ii}]{NIST:DLMF}. This increase occurs similarly in \eqref{Bound4} for the case $\ee=\theta$
except that the increase is larger. Finally, for the case $\ee=\lambda$, the small $z$-asymptotics has to break down
near $\arg z=\pm\pi$ because $\UqK{\lambda}{a}{b}{z}$ has poles along the negative real $z$ axis. Note the factor
$\frac1{\sin\xi}$ in \eqref{Bound5b}.


\section{Other Properties of $\Uq{a}{b}{z}$}\label{S:other}

All choices of $\Uq{a}{b}{z}=\UqK{\ee}{a}{b}{z}$ satisfy the following $q$-difference equations:
\begin{equation}
\begin{split}
    (1-b)\Uq{a}{qb}{z}+b\Uq{a}{b}{qz}-\Uq{a}{b}{z}&=0,\\
    b(1-a)\Uq{qa}{b}{z}-a(1-b)\Uq{a}{qb}{z}+(a-b)\Uq{a}{b}{z}&=0,\\
    (1-abqz)(1-aq)\Uq{q^2a}{b}{z}-(1+(1-a)q-(b-aq)qaz)\Uq{qa}{b}{z}&\\
    +q\Uq{a}{b}{z}&=0.
\end{split}
\end{equation}
The $q$-difference equation
\begin{equation}
     \Uq{a}{b}{z}-\Uq{a}{bq}{z/q}=(1-a)z q^{-1}\Uq{aq}{bq}{z/q^2},
\end{equation}
seems to give us the continued fraction
\begin{equation}\label{CF}
    \frac{\Uq{a}{b}{z}}{\Uq{a}{bq}{z/q}}=_{?}1+\frac{\alpha_1z}{1+}\frac{\alpha_2z}{1+}\cdots,\ \text{with}
    \  \alpha_{2n+1}=\frac{1-aq^n}{q^{2n+1}},\ \alpha_{2n}=\frac{1-bq^n}{q^{2n}}.
\end{equation}
If we take $c_n=1+\frac{\alpha_1z}{1+}\frac{\alpha_2z}{1+}\cdots\frac{\alpha_{n-1}z}{1+\alpha_n z}$
then it seems that both $c_1,c_3,c_5,\ldots$, and $c_2,c_4,c_6,\ldots$, converge, but not to the same
limit and neither limit equals the left-hand side of \eqref{CF} for any of
our three choices for $\UqK{\ee}{a}{b}{z}$.

Both sides of \eqref{CF} have the same asymptotic expansion as $z\to0$.
In the terminating case, that is, one of $a$, $b$ is of the form $q^{-n}$, equation \eqref{CF} holds.


\appendix

\section{Cauchy-Heine Integral Representations} \label{CauchyHeineIntRep}

Let $w_1(z)$ be a multivalued function with $w_1(z)=\bigO(1)$ as $z\to0$ in the sector $0\leq\arg z\leq2\pi$,
and having a simple connection relation $w_1(z\expe^{2\pi\iunit})-w_1(z)=2\pi\iunit Kw_2(z)$.
In applications, the constant $K$ is the Stokes multiplier and $w_2(z)$ is $q$-exponentially small as $z\to 0+$.
We take $0<\rho<|z|<R$, $\arg z\in(0,2\pi)$ and ${\mathcal C}$ the keyhole contour displayed in Figure \ref{CHcontour}. Then
\begin{equation} \label{CauchyHeine2Term}
\begin{split}
    w_1(z)&=\frac{1}{2\pi \iunit}\oint_{\mathcal{C}}\frac{w_1(t)}{t-z}\intd t
    =\frac{1}{2\pi\iunit}\left(\int_\rho^R+\int_{|t|=R}+\int_{R\expe^{2\pi\iunit}}^{\rho\expe^{2\pi\iunit}}
    +\int_{|t|=\rho}\right)\frac{w_1(t)}{t-z}\intd t\\
    &=\frac{1}{2\pi\iunit}\int_0^R \frac{w_1(t)-w_1(t\expe^{2\pi \iunit})}{t-z}\intd t
    +\frac{1}{2\pi\iunit}\int_{|t|=R}\frac{w_1(t)}{t-z}\intd t\\
   &=K\int_0^R\frac{w_2(t)}{z-t}\intd t
   +\frac{1}{2\pi\iunit}\int_{|t|=R}\frac{w_1(t)}{t-z}\intd t,
\end{split}
\end{equation}
where we took the limit $\rho\to 0$.
\begin{figure}[H]
\centering\includegraphics[width=0.5
\textwidth]{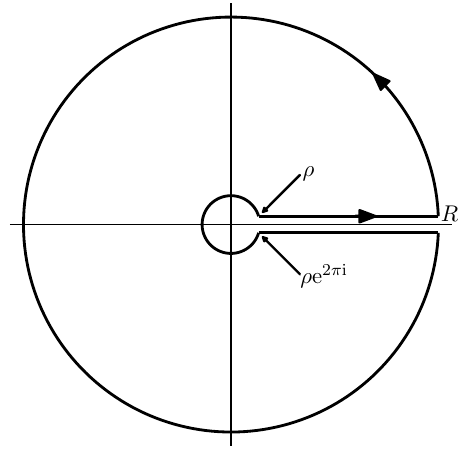}
\caption{Contour of integration $\mathcal{C}$ used to obtain the Cauchy-Heine integral representation.}
\label{CHcontour}
\end{figure}

\section{Proofs} \label{AS:Proofs}

\begin{proof}[Proof for Theorem \ref{thmCompareqL}]
First we show that the main results of Theorem \ref{thmCompareqL} hold for the case $S(-t)=1$. 
We use \eqref{PartFrac1} twice and partial fractions,
\begin{align}
    &(\LqK{\theta} B)(z)-(\LqK{\lambda} B)(z)=\sum_{n=-\infty}^\infty
    \left(\frac{-1}{\ln q}\int_{q^{n+1}}^{q^n} \frac{\intd t}{\Thetaq{t}(t+z)}-
    \frac{q^n\lambda}{\Thetaq{q^n\lambda}(q^n\lambda+z)}\right)\\
    &\qquad =\sum_{n=-\infty}^\infty
    \left(\frac{-1}{\ln q}\int_{q}^{1} \frac{q^n\intd t}{\Thetaq{q^nt}(q^nt+z)}-
    \frac{q^n\lambda}{\Thetaq{q^n\lambda}(q^n\lambda+z)}\right)\notag\\
    &\qquad =\sum_{n=-\infty}^\infty
    \left(\int_{0}^{1} \frac{q^{n+\tau}\intd \tau}{\Thetaq{q^{n+\tau}}(q^{n+\tau}+z)}-
    \frac{q^n\lambda}{\Thetaq{q^n\lambda}(q^n\lambda+z)}\right)\notag\\
    &\qquad =\sum_{n=-\infty}^\infty
    \sum_{m=-\infty}^\infty 
    \frac{\left(-1\right)^mq^{\binom{m}{2}}}{\qpr{q}{q}{\infty}^3}
    \left(\int_{0}^{1} \frac{q^\tau\intd \tau}{(q^\tau+q^{-m-n})(q^{n+\tau}+z)}-
    \frac{\lambda}{(\lambda+q^{-m-n})(q^n\lambda+z)}\right)\notag\\
    &\qquad =\sum_{n=-\infty}^\infty
    \sum_{m=-\infty}^\infty 
    \frac{\left(-1\right)^mq^{\binom{m}{2}}}{\qpr{q}{q}{\infty}^3(z-q^{-m})}
    \left(\int_{0}^{1} \frac{q^\tau\intd \tau}{q^\tau+q^{-m-n}}-
    \frac{\lambda}{\lambda+q^{-m-n}}\right)\notag\\
    &\qquad\qquad-\sum_{n=-\infty}^\infty
    \sum_{m=-\infty}^\infty 
    \frac{\left(-1\right)^mq^{\binom{m}{2}}}{\qpr{q}{q}{\infty}^3(z-q^{-m})}
    \left(\int_{0}^{1} \frac{q^\tau\intd \tau}{q^\tau+zq^{-n}}-
    \frac{\lambda}{\lambda+zq^{-n}}\right)\notag\\
    &\qquad =\sum_{n=-\infty}^\infty
    \sum_{m=-\infty}^\infty 
    \frac{\left(-1\right)^mq^{\binom{m}{2}}}{\qpr{q}{q}{\infty}^3(z-q^{-m})}
    \left(\int_{0}^{1} \frac{q^\tau\intd \tau}{q^\tau+q^{-n}}-
    \frac{\lambda}{\lambda+q^{-n}}\right)\notag\\
    &\qquad\qquad-\sum_{n=-\infty}^\infty
    \sum_{m=-\infty}^\infty 
    \frac{\left(-1\right)^mq^{\binom{m}{2}}}{\qpr{q}{q}{\infty}^3(z-q^{-m})}
    \left(\int_{0}^{1} \frac{q^\tau\intd \tau}{q^\tau+zq^{-n}}-
    \frac{\lambda}{\lambda+zq^{-n}}\right)\notag\\
    &\qquad =\frac1{\Thetaq{-z}}\sum_{n=-\infty}^\infty
    \left(\int_{0}^{1} \frac{q^\tau\intd \tau}{q^\tau+zq^{-n}}-\int_{0}^{1} \frac{q^\tau\intd \tau}{q^\tau+q^{-n}}
    +\frac{\lambda}{\lambda+q^{-n}}
    -\frac{\lambda}{\lambda+zq^{-n}}\right).\notag
\end{align}
Thus,
\begin{equation}\label{Pqdlong}
\begin{split}
P_q^{(d)}(z;\lambda)&=\sum_{n=-\infty}^\infty
    \left(\frac1{\ln q}\ln\left(\frac{q+zq^{-n}}{1+zq^{-n}}\right)
    -\frac1{\ln q}\ln\left(\frac{q+q^{-n}}{1+q^{-n}}\right)\right.\\
    &\qquad\qquad\qquad\qquad\qquad\qquad\left.+\frac{\lambda}{\lambda+q^{-n}}
    -\frac{\lambda}{\lambda+zq^{-n}}\right).
\end{split}\end{equation}
The sum of the first two terms gives a telescoping sum, and for the final two terms, we use 
\eqref{lemmaparfrac1}. In this way, we obtain the first representation of \eqref{Pqd1}. 
The second representation in terms of $\Pq{\tau}$ functions follows from using the relation
\begin{equation}\label{logdPq}
    \frac{t\dPq{t}}{\Pq{t}}=\frac{t\dThetaq{t}}{\Thetaq{t}}+\frac{\ln t}{\ln q}-\frac{1}{2},
\end{equation}
which is obtained by logarithmic differentiation of \eqref{qperiodic} and \eqref{Eq}.

When we apply the operator $f(\lambda)\mapsto \frac{-1}{\ln q}\int_{q}^1 \Pq{\frac{\lambda}{z}}
f(\lambda)\frac{\intd \lambda}{\lambda}$ to the left-hand side of \eqref{Compare2}, we obtain,
via \eqref{IntlamdaE}, the left-hand side of \eqref{Compare1}. Equation \eqref{Pqc0} follows.

We use the second representation of \eqref{Pqd1} in \eqref{Pqc0} and the fact that
$\Pq{q^t}$ is an even 1-periodic function in
\begin{equation}\label{IntPq1}
    \begin{split}
    P_q^{(c)}(z)&=\int_0^1\Pq{\frac{q^t}z}\frac{q^t\dPq{q^t}}{\Pq{q^t}}\intd t
    =-\int_0^1\Pq{zq^{-t}}\frac{q^{-t}\dPq{q^{-t}}}{\Pq{q^{-t}}}\intd t\\
    &=-\int_0^1\Pq{q^\tau}\frac{q^\tau\dPq{\frac{q^\tau}z}}{z\Pq{\frac{q^\tau}z}}\intd \tau
    =\frac{1}{z\ln q}\int_q^1\frac{\Pq{t}\dPq{\frac{t}{z}}}{\Pq{\frac{t}{z}}}\intd t,
    \end{split}
\end{equation}
which gives the second representation in \eqref{PqcPq}. Using \eqref{PqInt}, \eqref{logdPq} and \eqref{IntPq1}, we may obtain the first representation via
\begin{equation}
\begin{split}
    P_q^{(c)}(z)&=\frac{1}{z\ln q}\int_q^1\frac{\Pq{t}\dThetaq{\frac{t}{z}}}{\Thetaq{\frac{t}{z}}}\intd t-\frac{\ln z}{\ln^2q}\int_q^1\Pq{t}\frac{\intd t}{t}\\
    &\qquad\qquad\qquad+\frac{1}{\ln q}\int_q^1\Pq{t}\left(\frac{\ln t}{\ln q}-\frac{1}{2}\right)\frac{\intd t}{t}\\
    &=\frac{1}{z\ln q}\int_q^1\frac{\Pq{t}\dThetaq{\frac{t}{z}}}{\Thetaq{\frac{t}{z}}}\intd t+\frac{\ln z}{\ln q}-\int_{-\frac{1}{2}}^{\frac{1}{2}}\Pq{q^{\tau+\frac{1}{2}}}\tau\intd\tau.
\end{split}
\end{equation}
The final term in the above expression evaluates to zero as the integrand is odd.

To obtain series representation \eqref{Pqc1}, we present \eqref{IntPq1} as
\begin{equation}
    P_q^{(c)}(z)=\int_0^1q^t\dPq{q^t}\ln\left(\Pq{\frac{q^t}{z}}\right)\intd t,
\end{equation}
and use expansions \eqref{qperiodic} and \eqref{PqFourier}. In the resulting double series,
the integrals can be evaluated resulting in \eqref{Pqc1}.

The first equation of \eqref{Pqd2} is a direct consequence of \eqref{Pqd1}. The second equation of 
\eqref{Pqd2} is \eqref{lemmaparfrac3}.
Limits \eqref{Compare1a} and \eqref{Compare2a} follow from \eqref{PartFrac1}, \eqref{Pqc1} and 
\eqref{Pqdlong}.

Finally, for general $S(-t)$, we first make the observations that for $n=0,1,2,\ldots$, we have
\begin{equation}\begin{split}\label{uniform0}
    \frac{-1}{\ln q}\int_0^\infty \frac{(1-\Pq{t})t^n}{\Thetaq{t}}\intd t
    &=\frac{-1}{\ln q}\int_0^\infty \frac{t^n}{\Thetaq{t}}\intd t-
    \Cq\int_0^\infty\Eq{t}t^n\intd t\\
    &=q^{-\binom{n+1}{2}}-q^{-\binom{n+1}{2}}=0.
\end{split}\end{equation}
In the identity,
\begin{equation}\label{uniform1}
    (\LqK{\theta} B)(z)-(\LqK{E} B)(z)=\frac{S(z)}{\Thetaq{-z}}P_q^{(c)}(z)-
    \frac{1}{\ln q}\int_0^\infty \frac{(1-\Pq{t})(S(-t)-S(z))}{\Thetaq{t}(t+z)}\intd t,
\end{equation}
we will show that the final integral is zero.
The function $f(t)=\frac{S(-t)-S(z)}{t+z}$ is an entire function of order and type zero with the same growth
condition as $S(t)$. It follows that in the final integral of \eqref{uniform1}, we can expand $f(t)$ in its
Maclaurin series and we are allowed (Fubini's theorem) to interchange integration and summation. 
Because of \eqref{uniform0}, the resulting series is zero and \eqref{Compare1} holds.
The derivation of the general case of \eqref{Compare2} is very similar.
\end{proof}

\begin{proof}[Proof for Corollary \ref{crllPqcdStokes}]
From \eqref{Pqc1},
\begin{equation}
    P_q^{(c)}(z\expe^{2\pi\iunit})=-\frac{2\pi}{\ln q}\sum_{n=1}^\infty (-1)^n\qhat^{n^2}\frac{\sin\left(2\pi n\frac{\ln z}{\ln q}+2\iunit n\ln\qhat\right)}{\sinh(n\ln\qhat)}.
\end{equation}
Using trigonometric identities for complex argument, we have
\begin{equation}
\begin{split}
    \sin\left(2\pi n\frac{\ln z}{\ln q}+2\iunit n\ln\qhat\right)&=\sin\left(2\pi n\frac{\ln z}{\ln q}\right)\cosh(2n\ln\qhat)+\iunit\cos\left(2\pi n\frac{\ln z}{\ln q}\right)\sinh(2n\ln\qhat)\\
    &=\sin\left(2\pi n\frac{\ln z}{\ln q}\right)+2\sin\left(2\pi n\frac{\ln z}{\ln q}\right)\sinh^2(n\ln\qhat)\\
    &\qquad+2\iunit\cos\left(2\pi n\frac{\ln z}{\ln q}\right)\cosh(n\ln\qhat)\sinh(n\ln\qhat)\\
    &=\sin\left(2\pi n\frac{\ln z}{\ln q}\right)+2\iunit\sinh(n\ln\qhat)\cos\left(2\pi n\frac{\ln(z\expe^{\pi\iunit})}{\ln q}\right).
\end{split}
\end{equation}
Therefore,
\begin{equation}
\begin{split}
    P_q^{(c)}(z\expe^{2\pi\iunit})&=P_q^{(c)}(z)-\frac{4\pi\iunit}{\ln q}\sum_{n=1}^\infty(-1)^n\qhat^{n^2}\cos\left(2\pi n\frac{\ln(z\expe^{\pi\iunit})}{\ln q}\right),\\
    &=P_q^{(c)}(z)-\frac{2\pi\iunit}{\ln q}(\Pq{z\expe^{\pi\iunit}}-1).
\end{split}
\end{equation}
Connection formula \eqref{PqdStokes} is easily obtained from \eqref{Pqd1}.
\end{proof}

\begin{proof}[Proof for Theorem \ref{thmIntegrals}]
Using analytic continuation \cite[(1.13.13)]{KLS2010}, we obtain integral representation 
\eqref{Uq2phi1} from \eqref{2phi0BorelLaplaceK}.

Using \cite[(1.13.7)]{KLS2010}, we obtain \eqref{UqPoch} from \eqref{Uq2phi1} via
\begin{equation}
\begin{split}
    \UqK{\ee}{a}{b}{z}&=\qpr{a}{q}{\infty}\int_0^\infty
    \frac{\qpr{-bt}{q}{\infty}}{
    \qpr{-t}{q}{\infty}}\qphi{2}{1}{-t,0}{-bt}{q}{a}
    \K{\frac{t}{z}} \frac{\intd t}{t}\\
    &=\qpr{a}{q}{\infty}\int_0^\infty\sum_{n=0}^\infty
   \frac{\qpr{-bq^nt}{q}{\infty}}{
    \qpr{-q^nt}{q}{\infty}}\frac{a^n}{\qpr{q}{q}{n}}
    \K{\frac{t}{z}} \frac{\intd t}{t}\\
    &=\qpr{a}{q}{\infty}\int_0^\infty\sum_{n=0}^\infty
   \frac{\qpr{-bt}{q}{\infty}}{
    \qpr{-t}{q}{\infty}}\frac{a^n}{\qpr{q}{q}{n}}
    \K{\frac{t}{z}q^{-n}} \frac{\intd t}{t}\\
    &=\qpr{a}{q}{\infty}\int_0^\infty
    \frac{\qpr{-bt}{q}{\infty}}{
    \qpr{-t}{q}{\infty}}\qpr{-\frac{aqz}{t}}{q}{\infty}
    \K{\frac{t}{z}} \frac{\intd t}{t}.
\end{split}
\end{equation}

For \eqref{UqSymmetric}, we start with \eqref{UqPoch} and use
$\int_0^\infty=\sum_{n=-\infty}^\infty\int_{q^{n+1}}^{q^n}$:
\begin{equation}
\begin{split}
    \UqK{\ee}{a}{b}{z}&=\qpr{a}{q}{\infty}\int_q^1
    \frac{\qpr{-bt,-\frac{aqz}{t}}{q}{\infty}}{
    \qpr{-t}{q}{\infty}}
    \qpsi{2}{2}{-t,-\frac{t}{az}}{-bt,0}{q}{a}
    \K{\frac{t}{z}} \frac{\intd t}{t}\\
    &=\qpr{abz}{q}{\infty}\int_q^1
    \frac{\qpr{-at,-bt}{q}{\infty}}{
    \qpr{-t}{q}{\infty}}
    \qpsi{1}{2}{-t}{-at,-bt}{q}{-\frac{t}{z}}
    \K{\frac{t}{z}} \frac{\intd t}{t}\\
    &=\qpr{abz}{q}{\infty}\int_0^\infty
    \frac{\qpr{-at,-bt}{q}{\infty}}{
    \qpr{-t}{q}{\infty}}
    \K{\frac{t}{z}} \frac{\intd t}{t}.
\end{split}
\end{equation}
where we have used the limiting case $d\to0$ of
\cite[\href{http://dlmf.nist.gov/17.10.E1}{17.10.1}]{NIST:DLMF} in the second step.

Now we have three different integral representations for $\UqK{\KK}{a}{b}{z}$ and these can be used
to obtain connection formula \eqref{connectionE}. Hence, we have all the ingredients for the
construction of the Cauchy-Heine representation \eqref{UqCauchyHeine}: 
We consider the Cauchy-Heine representation in 
Appendix \ref{CauchyHeineIntRep} with $w_1(z)=\UqK{\KK}{a}{b}{z\expe^{-\pi\iunit}}$,
$w_2(z)=\K{qz}\qpr{-bqz}{q}{\infty}\qphi{1}{1}{\frac{q}{a}}{-bqz}{q}{-aqz}$, and
(use \eqref{connectionE}) connection formula $w_1(z\expe^{2\pi\iunit})-w_1(z)=-2\pi\iunit K_0 w_2(z)$.
It follows from connection formula \eqref{connectionK} with the constraint that $|a|,|b|<1$ that in the final line of
\eqref{CauchyHeine2Term}, we can let $R\to\infty$ and the contribution from infinity vanishes.
Thus,
\begin{equation}
    \UqK{\ee}{a}{b}{z\expe^{-\pi\iunit}}=-K_0\int_0^\infty\frac{\qpr{-bqt}{q}{\infty}\qphi{1}{1}{\frac{q}{a}}{-bqt}{q}{-aqt}}{z-t}\K{qt}\intd t,
\end{equation}
and therefore,
\begin{equation} \label{CauchyHeineK0}
    \UqK{\ee}{a}{b}{z}=K_0\int_0^\infty\frac{\qpr{-bqt}{q}{\infty}\qphi{1}{1}{\frac{q}{a}}{-bqt}{q}{-aqt}}{t+z}\K{qt}\intd t,
\end{equation}
with the constraint that $|a|,|b|<1$.
\end{proof}

\begin{proof}[Proof for Theorem \ref{thmSums}]
Sum representations \eqref{2phi0BorelLaplaceL} and \eqref{Lsum2} are the same sum, but with
analytic continuation \cite[(1.13.13)]{KLS2010} for the terms. Representation \eqref{Lsum3} follows from
\eqref{Lsum2} and transformation \cite[(1.13.7)]{KLS2010} via
\begin{equation}
    \begin{split}
        &\frac1{\qpr{a}{q}{\infty}}\sum_{n=-\infty}^\infty 
        \frac{\qpr{-bq^n\lambda}{q}{\infty}\qphi{1}{1}{b}{-bq^n\lambda}{q}{%
    -aq^n\lambda}}{\qpr{-q^n\lambda}{q}{\infty}\Thetaq{\frac{q^n\lambda}{z}}}=
    \sum_{n=-\infty}^\infty \frac{\qpr{-bq^n\lambda}{q}{\infty}\qphi{2}{1}{-q^n\lambda,0}{-bq^n\lambda}{q}{a}}{%
    \qpr{-q^n\lambda}{q}{\infty}\Thetaq{\frac{q^n\lambda}{z}}}\\
    &=\sum_{n=-\infty}^\infty \frac{\qpr{-bq^n\lambda}{q}{\infty}}{\qpr{-q^n\lambda}{q}{\infty} 
    \Thetaq{\frac{q^n\lambda}{z}}}\sum_{m=0}^\infty 
    \frac{\qpr{-q^n\lambda}{q}{m}a^m}{\qpr{-b q^n\lambda}{q}{m}\qpr{q}{q}{m}}\\
    &=\sum_{n=-\infty}^\infty\sum_{m=0}^\infty 
    \frac{\qpr{-bq^{n+m}\lambda}{q}{\infty}a^m}{\qpr{-q^{n+m}\lambda}{q}{\infty}\qpr{q}{q}{m} 
    \Thetaq{\frac{q^n\lambda}{z}}} 
    =\sum_{n=-\infty}^\infty\sum_{m=0}^\infty 
    \frac{\qpr{-bq^{n}\lambda}{q}{\infty}a^m}{\qpr{-q^{n}\lambda}{q}{\infty}\qpr{q}{q}{m} 
    \Thetaq{\frac{q^{n-m}\lambda}{z}}}\\
    &=\sum_{n=-\infty}^\infty\sum_{m=0}^\infty 
    \frac{\qpr{-bq^{n}\lambda}{q}{\infty}q^{\binom{m}{2}}\left(\frac{aqz}{q^n\lambda}\right)^m}{\qpr{-q^{n}\lambda}{q}{\infty}\qpr{q}{q}{m} 
    \Thetaq{\frac{q^{n}\lambda}{z}}}
    =\sum_{n=-\infty}^\infty\frac{\qpr{-bq^n\lambda}{q}{\infty}\qpr{-\frac{aqz}{q^n\lambda}}{q}{\infty}}{%
    \qpr{-q^n\lambda}{q}{\infty}\Thetaq{\frac{q^n\lambda}{z}}}.
    \end{split}
\end{equation}
To obtain \eqref{Lsum4}, we represent \eqref{Lsum3} as
\begin{equation}
    \begin{split}
        \eqref{Lsum3}&=\frac{\qpr{a,-b\lambda,-\frac{aqz}{\lambda}}{q}{\infty}}{%
        \qpr{-\lambda}{q}{\infty}\Thetaq{\frac{\lambda}{z}}}\sum_{n=-\infty}^\infty
        \frac{\qpr{-\lambda}{q}{n}\qpr{\frac{-\lambda}{az}}{q}{n}}{\qpr{-b\lambda}{q}{n}}\\
        &=\frac{\qpr{a,-b\lambda,-\frac{aqz}{\lambda}}{q}{\infty}}{%
        \qpr{-\lambda}{q}{\infty}\Thetaq{\frac{\lambda}{z}}}
        \qpsi{2}{2}{-\lambda,\frac{-\lambda}{az}}{-b\lambda,0}{q}{a}\\
        &=\frac{\qpr{abz,-a\lambda,-b\lambda}{q}{\infty}}{%
        \qpr{-\lambda}{q}{\infty}\Thetaq{\frac{\lambda}{z}}}
        \qpsi{1}{2}{-\lambda}{-a\lambda,-b\lambda}{q}{-\frac\lambda{z}}\\
        &=\frac{\qpr{abz,-a\lambda,-b\lambda}{q}{\infty}}{%
        \qpr{-\lambda}{q}{\infty}\Thetaq{\frac{\lambda}{z}}}
        \sum_{n=-\infty}^\infty
        \frac{\qpr{-\lambda}{q}{n}q^{\binom{n}{2}}\left(\frac\lambda{z}\right)^n}{%
        \qpr{-a\lambda,-b\lambda}{q}{n}}=\eqref{Lsum4},
    \end{split}
\end{equation}
where we have used the limiting case $d\to0$ of
\cite[\href{http://dlmf.nist.gov/17.10.E1}{17.10.1}]{NIST:DLMF} in the second step.

We present \eqref{Lsum3} as
\begin{equation}\label{Lsum3a}
    \UqK{\lambda}{a}{b}{z}
    =\frac{\qpr{a}{q}{\infty}}{\Thetaq{\frac{\lambda}{z}}}\sum_{n=-\infty}^\infty
    \frac{\qpr{-bq^n\lambda,-\frac{aqz}{q^n\lambda}}{q}{\infty}q^{\binom{n}{2}}
    \left(\frac{\lambda}{z}\right)^n}{\qpr{-q^n\lambda}{q}{\infty}}.
\end{equation}
It follows from \eqref{PartFrac1} that the right-hand side of \eqref{Lsum3a} has a pole at $z=-\lambda q^k$
with residue
\begin{equation}\label{ResidueL2}
    \begin{split}
        &\frac{\lambda\qpr{a}{q}{\infty}}{\qpr{q}{q}{\infty}^3}\sum_{n=-\infty}^\infty
        \frac{\qpr{-bq^n\lambda,a q^{k-n+1}}{q}{\infty}\left(-1\right)^{n+k+1}q^{\binom{n}{2}+\binom{k}{2}+(1-n)k}}{\qpr{-q^n\lambda}{q}{\infty}}\\
        &\qquad=\frac{\lambda q^k\qpr{a}{q}{\infty}}{\qpr{q}{q}{\infty}^3}\sum_{n=-\infty}^\infty
        \frac{\qpr{-bq^{n+k+1}\lambda,a q^{-n}}{q}{\infty}\left(-1\right)^{n}q^{\binom{n+1}{2}}}{\qpr{-q^{n+k+1}\lambda}{q}{\infty}}\\
         &\qquad=\frac{\lambda q^k\qpr{a}{q}{\infty}^2\qpr{-bq^{k+1}\lambda}{q}{\infty}}{\qpr{q}{q}{\infty}^3\qpr{-q^{k+1}\lambda}{q}{\infty}}\sum_{n=-\infty}^\infty
        \frac{\qpr{\frac{q}{a},-q^{k+1}\lambda}{q}{n}a^n}{\qpr{-b q^{k+1}\lambda}{q}{n}}\\
        &\qquad=\frac{\lambda q^k\qpr{a}{q}{\infty}^2\qpr{-bq^{k+1}\lambda}{q}{\infty}}{\qpr{q}{q}{\infty}^3\qpr{-q^{k+1}\lambda}{q}{\infty}}\qpsi{2}{2}{\frac{q}{a},-\lambda q^{k+1}}{-b\lambda q^{k+1},0}{q}{a}\\
        &\qquad=\frac{\lambda q^k\qpr{a,b,-bq^{k+1}\lambda}{q}{\infty}}{\qpr{q}{q}{\infty}^2
        \qpr{-q^{k+1}\lambda,-\frac{q^{-k}}\lambda}{q}{\infty}}
        \sum_{n=0}^\infty \frac{\qpr{\frac{q}{a}}{q}{n}q^{\binom{n}{2}}}{\qpr{-b q^{k+1}\lambda,q}{q}{n}}\left(a\lambda q^{k+1}\right)^n,
    \end{split}
\end{equation}
where we have used the limiting case $d\to0$ of
\cite[\href{http://dlmf.nist.gov/17.10.E1}{17.10.1}]{NIST:DLMF} in the final step.
Note that the final line in \eqref{ResidueL2} can be presented as
\begin{equation}\label{ResidueL1}
    \frac{\qpr{a,b}{q}{\infty}}{\qpr{q}{q}{\infty}}\frac{\qpr{-bq^{k+1}\lambda}{q}{\infty}\qphi{1}{1}{\frac{q}{a}}{-bq^{k+1}\lambda}{q}{-aq^{k+1}\lambda}}{\Thetaq{q^{k+1}\lambda}}q^k\lambda.
\end{equation}
Hence, \eqref{Lsum5} is just the partial fraction representation of \eqref{Lsum3}.
\end{proof}

\begin{proof}[Proof for Theorem \ref{thmMB}]
This theorem is a direct consequence of the Mellin--Barnes integral representation
\begin{equation}\label{phi21MB}
    \qphi{2}{1}{a,b}{0}{q}{-t}
    =\frac{-K_0}{2\pi\iunit}\int_{-\iunit\infty}^{\iunit\infty}
    \frac{\pi\qpr{q^{1+s}}{q}{\infty}t^s}{\sin(\pi s)\qpr{aq^s,bq^s}{q}{\infty}}
    \intd s,
\end{equation}
see \cite[\href{http://dlmf.nist.gov/17.6.E29}{17.6.29}]{NIST:DLMF}. 
We use \eqref{phi21MB} in \eqref{2phi0BorelLaplaceK} for the case $\ee=E$.
In the resulting double integral, the $t$-integral is just 
$\Cq \int_0^\infty t^{s-1}\Eq{\frac{t}{z}}\intd t=q^{-\binom{s}{2}}z^s$, compare \eqref{BLpower}. Hence,
we obtain Mellin--Barnes integral representation \eqref{MB1}.

The case $\ee=\theta$ is similar, but we have to evaluate
\begin{equation}\label{BLpowerGeneral}
    \begin{split}
        \frac{-1}{\ln q}\int_0^\infty \frac{t^{s-1}}{\Thetaq{\frac{t}{z}}}\intd t&=
        \frac{-z^s}{\ln q}\int_0^\infty \frac{t^{s-1}}{\Thetaq{t}}\intd t
        =\frac{-z^s}{\ln q}\sum_{n=-\infty}^\infty\frac{\left(-1\right)^n q^{\binom{n}{2}}}{\qpr{q}{q}{\infty}^3}
        \int_0^\infty \frac{t^{s-1}}{t+q^{-n}}\intd t\\
        &=\frac{-z^s}{\qpr{q}{q}{\infty}^3\ln q}\sum_{n=-\infty}^\infty\left(-1\right)^n q^{\binom{n}{2}-n(s-1)}
        \frac{\pi}{\sin(\pi s)}\\
        &=\frac{-z^s}{\qpr{q}{q}{\infty}^3\ln q}\frac{\pi}{\sin(\pi s)}\Thetaq{-q^{s}},
    \end{split}
\end{equation}
where we have used \eqref{PartFrac1}, \cite[\href{http://dlmf.nist.gov/5.12.E3}{5.12.3}]{NIST:DLMF} and \eqref{Thetaq}.
Hence, \eqref{MB2} follows. Note that \eqref{BLpowerGeneral} also follows from \cite[Eq.(1.7)]{Ismail:1995:SBB} by 
letting $a,b\to+\infty$.

Finally, we use \eqref{phi21MB} in \eqref{2phi0BorelLaplaceL} and obtain
\begin{equation}
    \UqK{\lambda}{a}{b}{z}
    =\frac{-K_0}{2\pi\iunit}\int_{-\iunit\infty}^{\iunit\infty}
    \frac{\pi\qpr{q^{1+s}}{q}{\infty}}{%
    \sin(\pi s)\qpr{aq^s,bq^s}{q}{\infty}}
    \sum_{n=-\infty}^\infty 
    \frac{q^{ns}\lambda^s}{\Thetaq{\frac{q^n\lambda}{z}}}
    \intd s.
\end{equation}
The sum can be evaluated,
\begin{equation}\label{BLpowerLambda}
    \sum_{n=-\infty}^\infty 
    \frac{q^{ns}\lambda^s}{\Thetaq{\frac{q^n\lambda}{z}}}=\frac{\lambda^s}{\Thetaq{\frac{\lambda}{z}}}
    \sum_{n=-\infty}^\infty q^{\binom{n}{2}+ns}\left(\frac\lambda{z}\right)^n=
    \frac{\lambda^s\Thetaq{\frac{\lambda}{z}q^s}}{\Thetaq{\frac{\lambda}{z}}},
\end{equation}
and \eqref{MB3} follows.
\end{proof}

\begin{proof}[Proof for Theorem \ref{thmConnection}]
Connection relation \eqref{connectionK} for the case $\ee=\lambda$ with multipliers \eqref{pKlambda} 
follows from \cite{Zhang2002}. We will need the first property of \eqref{Fprop} and \eqref{Eqab}.
For an alternative proof, we use Mellin--Barnes integral representation \eqref{MB3}.
The residue contribution of the pole at $s=-\frac{\ln a}{\ln q}-m-\frac{2n\pi\iunit}{\ln q}$ is
\begin{equation}
    \frac{-\pi\Eq{a}\Thetaq{-a}\Pq{\frac{\lambda}{az}}}{\qpr{q}{q}{\infty}^3 q^{\frac18}\ln (q)\Pq{\frac{\lambda}{z}}}
    \frac{\qpr{b}{q}{\infty}}{\qpr{\frac{b}{a}}{q}{\infty}}\left(az\right)^{-\frac{\ln a}{\ln q}}
    \frac{\qpr{a}{q}{m}\left(\frac{q}{abz}\right)^m}{\qpr{\frac{aq}{b}}{q}{m}\qpr{q}{q}{m}}
    \frac{\lambda^{-\frac{2\pi\iunit n}{\ln q}}}{\sin\left(\frac\pi{\ln q}(2n\pi\iunit+\ln a)\right)}.
\end{equation}
We push the contour of integration in \eqref{MB3} to the left and obtain \eqref{connectionK} for the case
$\ee=\lambda$ with 
\begin{equation}
    \pK{\lambda}{a}{z}=
    \frac{-\pi\Eq{a}\Thetaq{-a}\Pq{\frac{\lambda}{az}}}{\qpr{q}{q}{\infty}^3 q^{\frac18}\ln (q)\Pq{\frac{\lambda}{z}}}
    \sum_{n=-\infty}^\infty
    \frac{\lambda^{-\frac{2\pi\iunit n}{\ln q}}}{\sin\left(\frac\pi{\ln q}(2n\pi\iunit+\ln a)\right)}.
\end{equation}
Via \eqref{PartFrac2a}, we can identify the infinite sum and obtain \eqref{pKlambda}.

For the case $\ee=E$, we consider the residue contribution of Mellin--Barnes integral representation \eqref{MB1}.
The details are very similar and we obtain \eqref{connectionK} for the case $\ee=E$ with 
\begin{equation}\label{pKE2}
    \pK{E}{a}{z}=
    \frac{-\pi\Eq{a}\Thetaq{-a}}{\qpr{q}{q}{\infty}^3 q^{\frac18}\ln q}
    \sum_{n=-\infty}^\infty
    \frac{\left(-1\right)^n\qhat^{n^2}\left(az\right)^{-\frac{2\pi\iunit n}{\ln q}}}{\sin\left(\frac\pi{\ln q}(2n\pi\iunit+\ln a)\right)}.
\end{equation}
This can also be presented as \eqref{pKE}.

Finally, for the case $\ee=\theta$, we obtain via a similar route,
\begin{equation}\label{pKtheta2}
    \pK{\theta}{a}{z}=
    \frac{\pi^2\theta_q^2(-a)a^{\frac{\ln a}{\ln q}}}{a\qpr{q}{q}{\infty}^6 \ln^2q}
    \sum_{n=-\infty}^\infty
    \frac{z^{-\frac{2\pi\iunit n}{\ln q}}}{\sin^2\left(\frac\pi{\ln q}(2n\pi\iunit+\ln a)\right)}.
\end{equation}
When we consider the integral $\frac{\pi^2}{2\pi\iunit}\int\frac{z^s}{\left(1-aq^s\right)\sin^2(\pi s)}\intd s$,
we obtain the identity
\begin{equation}\label{PartFrac2b}
    \frac{\pi^2}{a\ln^2 q}\sum_{n=-\infty}^\infty
    \frac{z^{\frac{-1}{\ln q}(2n\pi\iunit +\ln a)}}{\sin^2\left(\frac\pi{\ln q}(2n\pi\iunit+\ln a)\right)}
    =\frac{\ln z}{a\ln q}\sum_{n=-\infty}^\infty\frac{z^n}{1-aq^n}
    +\sum_{n=-\infty}^\infty\frac{\left(qz\right)^n}{\left(1-aq^n\right)^2}.
\end{equation}
The details are very similar to the proof of \eqref{PartFrac2a}. The first sum on the right-hand side of
\eqref{PartFrac2b} can be identified via \eqref{PartFrac2} and when we take the $a$-derivative of
this result, we identify the final sum in \eqref{PartFrac2b}. In this way, we obtain
\begin{equation}\label{pKtheta3}
    \pK{\theta}{a}{z}=\left(az\right)^{\ln(a)/\ln(q)}
    \frac{\Thetaq{-az}\Thetaq{-a}}{
    \Thetaq{-z}\qpr{q}{q}{\infty}^3}\left(
    \frac{\ln z}{a\ln q}
    +\frac{\dThetaq{-a}}{\Thetaq{-a}}
    -z\frac{\dThetaq{-az}}{\Thetaq{-az}}
    \right).
\end{equation}
Using \eqref{PqMinus}, we can present this as \eqref{pKtheta1}.
\end{proof}

\begin{proof}[Proof for Theorem \ref{thmErrorbounds}]
We obtain from \cite[Eq.(2.3)]{Joshi2025} and the observation
that $\Eq{t}\Thetaq{t}=\mathcal{O}(1)$ as $t\to\infty$ that
\begin{equation}\label{growthEtheta}
    \qpr{-cqt}{q}{\infty}\Eq{qt}=t^{-\frac{\ln c}{\ln q}}\bigO(1),\qquad\frac{\qpr{-cqt}{q}{\infty}}{\Thetaq{qt}}=t^{-\frac{\ln c}{\ln q}}\bigO(1),
\end{equation}
as $t\to\infty$.

Take $N$ a positive integer. Cauchy-Heine integral representation \eqref{Bound0} gives
\begin{equation} \label{Bound2a}
    \UqK{\ee}{a}{b}{z}=\sum_{n=0}^{N-1}\frac{\qpr{a,b}{q}{n}}{\qpr{q}{q}{n}}q^{-\binom{n}{2}}
    \left(-z\right)^n+R_{N,\ee}(a,b;z),
\end{equation}
with
\begin{equation} \label{Bound3a}
    R_{N,\ee}(a,b;z)=\left(-z\right)^N K_0\int_0^\infty\frac{t^{-N}u(t)}{t+z}\K{qt}\intd t.
\end{equation}
It follows from \eqref{growthEtheta} that this integral representation is valid  for $N>1-\frac{\ln c}{\ln q}$.

In the case that $\Re z\geq 0$, we have for $t>0$ that $\left|t+z\right|^{-1}\leq t^{-1}$. Thus
\begin{equation}
    \begin{split}
        \left|\int_0^\infty\frac{t^{-N}u(t)}{t+z}\K{qt}\intd t\right|&\leq
        \int_0^\infty\frac{t^{-N}|u(t)|}{|t+z|}\K{qt}\intd t\\
        &\leq M_q(a,b)\int_0^\infty t^{-N}\qpr{-cqt}{q}{\infty}\K{t}\intd t\\
        &=\frac{M_q(a,b)q^{-\binom{N}{2}}}{\qpr{cq^N}{q}{\infty}},
    \end{split}
\end{equation}
where we have used entries (1) and (5) of Table \ref{tab:qLaplace}.

In the sector $\arg z\in(\frac\pi2,\pi)$, we write $z=|z|\expe^{\iunit(\zeta+\frac\pi2)}$, with
$\zeta\in (0,\frac\pi2)$. First we deal with the case $\K{t}=\Cq\Eq{t}$,
\begin{equation}
    \begin{split}
        &\left|\Cq\int_0^{\infty\expe^{\iunit\zeta}}\frac{t^{-N}u(t)}{t+z}\Eq{qt}\intd t\right|
        = \left|\Cq\int_0^{\infty}\frac{t^{-N}u(t\expe^{\iunit\zeta})}{t+z\expe^{-\iunit\zeta}}\Eq{qt\expe^{\iunit\zeta}}\intd t\right|\\
        &\qquad\qquad\qquad\leq\Cq\exp\left({\tfrac{-\zeta^2}{2\ln q}}\right)\int_0^\infty\frac{t^{-N}|u(t\expe^{\iunit\zeta})|}{|t+z\expe^{-\iunit\zeta}|}\Eq{qt}\intd t\\
        &\qquad\qquad\qquad\leq M_q(a,b)\exp\left({\tfrac{-\zeta^2}{2\ln q}}\right)\Cq\int_0^\infty t^{-N}\qpr{-cqt}{q}{\infty}\Eq{t}\intd t\\
        &\qquad\qquad\qquad=\frac{M_q(a,b)\exp\left({\frac{-\zeta^2}{2\ln q}}\right)q^{-\binom{N}{2}}}{\qpr{cq^N}{q}{\infty}}.
    \end{split}
\end{equation}
When $\K{t}=\frac{-1}{\ln(q)\Thetaq{t}}$,
\begin{equation}
    \begin{split}
        &\left|\frac{-1}{\ln q}\int_0^{\infty\expe^{\iunit\zeta}}\frac{t^{-N}u(t)}{t+z}\frac{1}{\Thetaq{qt}}\intd t\right|
        =\left|\frac{-1}{\ln q}\int_0^{\infty}\frac{t^{-N}u(t\expe^{\iunit\zeta})}{t+z\expe^{-\iunit\zeta}}\frac{1}{\Thetaq{qt\expe^{\iunit\zeta}}}\intd t\right|\\
        &\qquad\qquad\leq\frac{-1}{\ln q}\int_0^\infty\frac{t^{-N}|u(t\expe^{\iunit\zeta})|}{|t+z\expe^{-\iunit\zeta}|}\frac{1}{|\Thetaq{qt\expe^{\iunit\zeta}}|}\intd t\\
        &\qquad\qquad=\Cq\int_0^\infty\frac{t^{-N}|u(t\expe^{\iunit\zeta})|}{|t+z\expe^{-\iunit\zeta}|}\frac{|\Eq{qt\expe^{\iunit\zeta}}|}{|\Pq{qt\expe^{\iunit\zeta}}|}\intd t\\
        &\qquad\qquad\leq M_q(a,b)\exp\left({\tfrac{-\zeta^2}{2\ln q}}\right)\Cq\int_0^\infty t^{-N}\qpr{-cqt}{q}{\infty}\frac{\Eq{t}}{|\Pq{t\expe^{\iunit\zeta}}|}\intd t.
    \end{split}
\end{equation}
We use \eqref{PqReciprocal} to obtain a bound for $1/|\Pq{t\expe^{\iunit\zeta}}|$. We have
\begin{equation}
    \left|\cos\left(\frac{2\pi n}{\ln q}(\iunit\zeta+\ln t)\right)\right|
    \leq\expe^{-\frac{2\pi n|\zeta|}{\ln q}}
    \leq\expe^{-\frac{\pi^2n}{\ln q}}=\qhat^{-\frac{n}{2}}\quad\text{for}\quad\zeta\in\left(-\tfrac{\pi}{2},\tfrac{\pi}{2}\right),
\end{equation}
and note that $\widetilde{a}_n\leq1$ since the terms in expansion \eqref{tan} are alternating and monotonically decreasing in magnitude. Using these, we have
\begin{equation}
\begin{split}
    \frac{1}{|\Pq{t\expe^{\iunit\zeta}}|}&\leq\frac{1}{\qpr{\qhat^2}{\qhat^2}{\infty}^3}\left(|\widetilde{a}_0|+2\sum_{n=1}^\infty|\widetilde{a}_n|\qhat^n
    \left|\cos\left(\frac{2\pi n}{\ln q}(\iunit\zeta+\ln t)\right)\right|\right)\\
    &\leq\frac{1}{\qpr{\qhat^2}{\qhat^2}{\infty}^3}\left(1+2\sum_{n=1}^\infty\qhat^{\frac{n}{2}}\right)
    =\frac{1}{\qpr{\qhat^2}{\qhat^2}{\infty}^3}\frac{1+\sqrt{\qhat}}{1-\sqrt{\qhat}}.
\end{split}
\end{equation}
Therefore,
\begin{equation}
\begin{split}
    &\left|\frac{-1}{\ln q}\int_0^{\infty\expe^{\iunit\zeta}}\frac{t^{-N}u(t)}{t+z}
    \frac{1}{\Thetaq{qt}}\intd t\right|\\
    &\qquad\leq\frac{M_q(a,b)\exp\left({\frac{-\zeta^2}{2\ln q}}\right)}{\qpr{\qhat^2}{\qhat^2}{\infty}^3}\left(\frac{1+\sqrt{\qhat}}{1-\sqrt{\qhat}}\right)\Cq\int_0^\infty t^{-N}\qpr{-cqt}{q}{\infty}\Eq{t}\intd t\\
    &\qquad=\frac{M_q(a,b)\exp\left({\frac{-\zeta^2}{2\ln q}}\right)q^{-\binom{N}{2}}}{\qpr{cq^N}{q}{\infty}\qpr{\qhat^2}{\qhat^2}{\infty}^3}\left(\frac{1+\sqrt{\qhat}}{1-\sqrt{\qhat}}\right).
\end{split}
\end{equation}
In the case that $\arg z\in(-\pi,-\frac\pi2)$, we write $z=|z|\expe^{\iunit(\zeta-\frac\pi2)}$, with
$\zeta\in (-\frac\pi2,0)$ and obtain the same upper-bounds for the cases $\K{t}=\Cq\Eq{t}$
and $\K{t}=\frac{-1}{\ln(q)\Thetaq{t}}$.

Finally, for the case $\ee=\lambda$, we denote \eqref{Lsum5} as
\begin{equation} \label{Bound4a}
    \UqK{\lambda}{a}{b}{z}=K_0\sum_{m=-\infty}^\infty\frac{u(q^m\lambda)}{\Thetaq{q^{m+1}\lambda}}\frac{q^m\lambda}{q^m\lambda+z},
\end{equation}
Take $N$ a positive integer. We obtain from \eqref{Bound4a} that
\begin{equation} \label{Bound5}
    \UqK{\lambda}{a}{b}{z}=\sum_{n=0}^{N-1}\frac{\qpr{a,b}{q}{n}}{\qpr{q}{q}{n}}q^{-\binom{n}{2}}
    \left(-z\right)^n+R_{N,\lambda}(a,b;z),
\end{equation}
with
\begin{equation} \label{Bound6}
    R_{N,\lambda}(a,b;z)=\left(-z\right)^N K_0\sum_{m=-\infty}^\infty
    \frac{\left(q^m\lambda\right)^{-N}u(q^m\lambda)}{\Thetaq{q^{m+1}\lambda}}\frac{q^m\lambda}{q^m\lambda+z}.
\end{equation}
In the case that $\Re z\geq 0$, we have for $\lambda>0$ that $|q^n\lambda|/|q^n\lambda+z|\leq1$. Thus
\begin{equation}
\begin{split}
    &\left|\sum_{m=-\infty}^\infty\frac{(q^m\lambda)^{-N}u(q^m\lambda)}{\Thetaq{q^{m+1}\lambda}}\frac{q^m\lambda}{q^m\lambda+z}\right|\leq\sum_{m=-\infty}^\infty\frac{|q^m\lambda|^{-N}|u(q^m\lambda)|}{\Thetaq{q^{m+1}\lambda}}\frac{|q^m\lambda|}{|q^m\lambda+z|}\\
    &\qquad\qquad\qquad\qquad\qquad\qquad\leq\sum_{m=-\infty}^\infty\frac{(q^m\lambda)^{-N}|u(q^m\lambda)|}{\Thetaq{q^{m+1}\lambda}}\\
    &\qquad\qquad\qquad\qquad\qquad\qquad\leq M_q(a,b)\sum_{m=-\infty}^\infty\frac{(q^m\lambda)^{-N}\qpr{-cq^{m+1}\lambda}{q}{\infty}}{\Thetaq{q^{m+1}\lambda}}\\
    &\qquad\qquad\qquad\qquad\qquad\qquad=\frac{M_q(a,b)q^{-\binom{N}{2}}}{\qpr{cq^N}{q}{\infty}}.
\end{split}
\end{equation}
Finally, when $\left|\arg z\right|\in(\frac\pi2,\pi)$, we take $z=|z|\expe^{\iunit\xi}$.
We observe that the function $\left|\frac{r}{r+z}\right|^2=\frac{r^2}{r^2+|z|^2+2r|z||\cos\xi|}$
has a maximum when $|z|=-r\cos\xi$. It follows that 
$\left|\frac{r}{r+z}\right|\leq \frac{1}{|\sin\xi|}$. We use this in
\begin{equation}
\begin{split}
    &\left|\sum_{m=-\infty}^\infty\frac{(q^m\lambda)^{-N}u(q^m\lambda)}{\Thetaq{q^{m+1}\lambda}}\frac{q^m\lambda}{q^m\lambda+z}\right|\leq\sum_{m=-\infty}^\infty\frac{|q^m\lambda|^{-N}|u(q^m\lambda)|}{\Thetaq{q^{m+1}\lambda}}\frac{|q^m\lambda|}{|q^m\lambda+z|}\\
    &\qquad\qquad\qquad\qquad\qquad\qquad\leq \frac{M_q(a,b)}{|\sin\xi|}\sum_{m=-\infty}^\infty\frac{(q^m\lambda)^{-N}\qpr{-cq^{m+1}\lambda}{q}{\infty}}{\Thetaq{q^{m+1}\lambda}}\\
    &\qquad\qquad\qquad\qquad\qquad\qquad=\frac{M_q(a,b)q^{-\binom{N}{2}}}{\qpr{cq^N}{q}{\infty}|\sin\xi|}.
\end{split}
\end{equation}
\end{proof}



\end{document}